\documentclass[11pt]{amsart}
\newfont{\cyr}{wncyr10 scaled 1100}
\usepackage{amsthm,amssymb,amsmath,amsfonts,mathrsfs,amscd, graphics}
\usepackage{mathtools}
\usepackage[latin1]{inputenc}

\usepackage[top=1.3in, bottom=1.3in, left=1.2in, right=1.2in]{geometry}
\usepackage[all]{xy}
\usepackage{latexsym}
\usepackage{longtable}
\usepackage{color}
\usepackage{tikz}
\usepackage{tikz-cd}
\usepackage{lua-visual-debug}
\usepackage{dsfont}
\usetikzlibrary{matrix}
\usepackage{hyperref}
\usepackage{enumitem}
\usepackage{setspace}

\newcommand*\ZZ{|[draw,circle]| \Z_2}
\numberwithin{equation}{section}
\usepackage[OT2,T1]{fontenc}
\DeclareSymbolFont{cyrletters}{OT2}{wncyr}{m}{n}
\DeclareMathSymbol{\Sha}{\mathalpha}{cyrletters}{"58}

\theoremstyle{plain}
\newtheorem{theorem}{Theorem}[section]
\newtheorem*{theorem*}{Theorem}

\newtheorem{lemma}[theorem]{Lemma}
\newtheorem{proposition}[theorem]{Proposition}

\numberwithin{equation}{section}

\theoremstyle{definition}
\newtheorem{definition}[theorem]{Definition}

\newtheorem{assumption}[theorem]{Assumption}

\theoremstyle{remark}
\newtheorem{obswr}[theorem]{Observation}
\newtheorem{remarkwr}[theorem]{Remark}

\newtheorem{intro-definition}[theorem]{Definition}

\newenvironment{remark}{\begin{remarkwr}\begin{upshape}}{\end{upshape}\end{remarkwr}}
\newenvironment{myproof}[2] {\paragraph{\emph{Proof of {#1} {#2} }}}{\hfill$\square$}

\def\Gal{\mathrm{Gal}}
\def\GL{\mathrm{GL}}

\def\GSpin{\mathrm{GSpin}}

\def\Nilp{(\mathrm{Nilp})}
\def\det{\mathrm{det}}
\def\Lie{\mathrm{Lie}}

\def\loc{\mathrm{loc}}
\def\ord{\mathrm{ord}}

\def\Ad{\mathrm{Ad}}

\def\new{\mathrm{new}}
\def\ac{\mathrm{ac}}
\def\frakm{\mathfrak{m}}

\def\interX{\mathfrak{X}}
\def\ss{\mathrm{ss}}
\def\ssp{\mathrm{ssp}}

\def\ac{\mathrm{ac}}

\def\mix{\mathrm{mix}}

\def\ram{\mathrm{ram}}

\def\Dr{\mathrm{Dr}}
\def\As{\mathrm{As}}

\DeclareMathOperator{\Hom}{Hom}

\DeclareMathOperator{\End}{End}

\def\calD{\mathcal{D}}

\def\calH{\mathcal{H}}
\def\calL{\mathcal{L}}
\def\calN{\mathcal{N}}
\def\calM{\mathcal{M}}
\def\calO{\mathcal{O}}
\def\calP{\mathcal{P}}

\def\calS{\mathcal{S}}

\def\calZ{\mathcal{Z}}

\def\frakd{\mathfrak{d}}

\def\frakn{\mathfrak{n}}

\def\Adel{\mathbf{A}}

\def\CC{\mathbf{C}}

\def\FF{\mathbf{F}}

\def\PP{\mathbf{P}}
\def\QQ{\mathbf{Q}}

\def\TT{\mathbb{T}}
\def\XX{\mathbb{X}}
\def\ZZ{\mathbf{Z}}

\def\rmA{\mathrm{A}}
\def\rmB{\mathrm{B}}
\def\rmC{\mathrm{C}}
\def\rmD{\mathrm{D}}
\def\rmF{\mathrm{F}}
\def\rmG{\mathrm{G}}
\def\rmH{\mathrm{H}}
\def\rmI{\mathrm{I}}
\def\rmE{\mathrm{E}}
\def\rmK{\mathrm{K}}

\def\rmM{\mathrm{M}}
\def\rmN{\mathrm{N}}

\def\rmU{\mathrm{U}}
\def\rmV{\mathrm{V}}
\def\rmR{\mathrm{R}}
\def\rmS{\mathrm{S}}
\def\rmQ{\mathrm{Q}}
\def\rmT{\mathrm{T}}
\def\rmZ{\mathrm{Z}}
\def\rmX{\mathrm{X}}

\def\rmW{\mathrm{W}}

\newcommand{\Iw}{\mathrm{Iw}}

\newcommand{\Spf}[1]{\mathrm{Spf} (#1)}

\include{thebibliography}

\begin{document}

\title[Hilbert--Blumenthal surfaces and distinguished periods]
{Flach system on Quaternionic Hilbert--Blumenthal surfaces and distinguished periods}

\author{Haining Wang }
\address{\parbox{\linewidth}{address:\\Shanghai Center for Mathematical Sciences,\\ Fudan University,\\No,2005 Songhu Road,\\Shanghai,200438, China.~ }}
\email{wanghaining1121@outlook.com}

\begin{abstract}
We study arithmetic properties of certain quaternionic periods of Hilbert modular forms arising from base change of elliptic modular forms. These periods which we call the distinguished periods are closely related to the notion of distinguished representation that appear in work of Harder--Langlands--Rapoport, Lai, Flicker--Hakim on the Tate conjectures for the Hilbert--Blumenthal surfaces and their quaternionic analogues. In particular, we prove an integrality result on the ratio of the distinguished period and the quaternionic Petersson norm associated to the modular form. Our method is based on an Euler system argument initiated by Flach by producing elements in the motivic cohomologies of the quaternionic Hilbert--Blumenthal surfaces with control of their ramification behaviours. 
\end{abstract}
   
\date{\today}

\maketitle
\tableofcontents
\section{Introduction}
In this article, we study arithmetic properties of certain integral periods of quaternionic Hilbert modular forms arising from base change.  These periods are closely related to the notion of distinguished representation which play a prominent role in the work of Harder--Langlands--Rapoport \cite{HLR}, Lai \cite{Lai}, Flicker--Hakim \cite{FH} on the Tate conjectures for the Hilbert--Blumenthal surfaces and their quaternionic analogues. We will therefore call these periods the distinguished periods. While the representation theoretic properties of these distinguished periods have been amply studied in the literature and have also been generalized and studied for higher rank groups using the relative trace formula, the arithmetic properties seem to be less well understood. We will provide some results in this direction below. More precisely, we will compare the distinguished period with another well-known period, namely the Petersson norm of the associated definite quaternionic modular form. We will show that the $\lambda$-adic valuation of the distinguished period will be greater than or equal to the  $\lambda$-adic valuation of  the Petersson norm for a prime $\lambda$ of the Hecke field over a fixed rational prime $l$. It is well-known that  the Petersson norm of a definite quaternionic modular form can be related to congruence module of the quaternionic modular form and hence our result shows that the distinguished period also captures this information. The distinguished period should also contain the information of the base change congruence module studied by Hida \cite{Hida} and Urban--Tilouine \cite{UT}, but we will address this question in another occasion.

This result should be viewed as an analogue of the integrality result of ratio of Petersson norms of modular and indefinite quaternionic modular forms proved by Prasana in \cite{pras} where he also shows that the positive valuation part of this ratio should be related to the local Tamagawa factors which measures quantitatively the level lowering congruences of the modular form. Our result should also be reminiscent of the result of Ribet--Takahashi \cite{RT} comparing the degree of the modular parametrization of elliptic curve by a classical modular curve and the degree of the parametrization of the elliptic curve by an indefinite Shimura curve. The results of Prasana have been extended to a program to understand of ratio of Petersson norms of general quaternionic Hilbert modular forms, \cite{IP1}, \cite{IP2}. It would be interesting to see if our method could be useful in this program.

The method to study these distinguished periods is based on the so-called Flach system of Jacquet--Langlands type introduced in the companion article \cite{Wang}. We will realize the $\lambda$-adic valuation of the distinguished period as a natural bound for the length of certain subspace of the Bloch--Kato Selmer group of the middle degree cohomology group of a suitable quaternionic Hilbert--Blumenthal surface. This subspace turns out to be isomorphic to the Bloch--Kato Selmer group of the adjoint Galois representation attached to the definite quaternionic modular form whose base-change defines the distinguished period. The length of the latter Selmer group can be determined using the Taylor--Wiles method and is equal to the $\lambda$-adic valuation of the Petersson norm of the definite quaternionic modular form and hence is less than equal to the $\lambda$-adic valuation of the distinguished period. This shows the desired integrality result of the ratio of periods as mentioned above. We remark that Flach system of Jacquet--Langlands type is a simplified geometric Euler system in the sense of \cite{Weston3} that should reflect Jacquet--Langlands correspondence in the cohomologies of Shimura varieties. In \cite{Wang}, this is reflected using arithmetic level raising results on the product of two Shimura curves. In this article, this is realized using the Tate conjecture for the special fiber of the quaternionic Hilbert--Blumenthal surface proved in  \cite{TX}, see also \cite{Lan}.  The version of the Tate conjectures for special fibers of Shimura varieties needed to construct similar Flach systems as in this article have been proved for a large class of Shimura varieties in \cite{TX} and \cite{XZ}, therefore we expect the strategy in this article and \cite{Wang} should be useful to study the integrality questions of ratio of periods coming from definite Shimura sets for other reductive groups. In a very similar setting, we will construct a Flach system of Jacquet--Langlands type on Picard modular surfaces and study the period of automorphic forms on $\rmU(3)$ coming from theta lifts of automorphic forms of $\rmU(2)$. We will also consider distinguished period arising from base change to a real quadratic extension 
of a unitary group $\rmU(n)$, our method seems to also work in this case and we will pursue this in another occasion.

It should also be interesting to note that usually automorphic periods  have good connections to special values of $L$-functions and Selmer groups can be studied in terms of these special values through Euler system arguments relating these automorphic periods and hence the special values to the size of the Selmer group via the so-called reciprocity formula. Here we turn things around and use known bounds for Selmer groups to study integral periods.

 \subsection{Main results} We will introduce some notations before we state our main results. Let $l\geq 5$ be a fixed prime. Let $f$ be a normalized newform in $\rmS_{2}(\Gamma_{0}(\rmN))$ whose associated automorphic representation of $\GL_{2}(\Adel)$ is denoted by $\pi$. We will assume that $\rmN$ admits a decomposition $\rmN=\rmN^{+}\rmN^{-}$ with $(\rmN^{+}, \rmN^{-})$ such that $\rmN^{-}$ is square-free and consists of odd number of prime factors. Let $\overline{\rmB}$ be the definite quaternion algebra of discriminant $\rmN^{-}$. Then $f$ admits a normalized Jacquet--Langlands transfer $f^{\dagger}$ to an automorphic form for the group $\rmG(\overline{\rmB})=\overline{\rmB}^{\times}$ whose associated automorphic representation will be denoted by $\pi^{\circ}=\pi^{\overline{\rmB}}$. We will be concerned with the base change $\phi^{\dagger}$ of $f^{\dagger}$ by a real quadratic field $\rmF$. Let $\rmE=\QQ(f)$ be a Hecke field of $f$ and $\iota_{l}: \overline{\QQ}\hookrightarrow \CC_{l}$ be a fixed embedding which defines a place $\lambda$ in $\rmE$. We will denote by $\rmE_{\lambda}$ the completion of $\rmE$ at $\lambda$ and by $\calO_{\lambda}$ its valuation ring. We will also denote by $\lambda$ the maximal ideal of $\calO_{\lambda}$ and will fix a uniformizer $\varpi$ for the ideal $\lambda$. Let $k_{\lambda}$ be the residue field of $\calO_{\lambda}$.
 
 Let $\overline{\rmQ}=\overline{\rmB}\otimes\rmF$ be the definite quaternion algebra obtained by base change to $\rmF$ and let $\rmG(\overline{\rmQ})$ be the algebraic group defined by the Weil restriction of $\overline{\rmQ}^{\times}$ from $\rmF$ to $\QQ$. Then $\phi^{\dagger}$ is contained in the automorphic representation $\pi^{\overline{\rmQ}}$ of $\rmG(\overline{\rmQ})$ obtained by base change of $\pi^{\circ}$ to $\rmF$. The automorphic forms $\phi^{\dagger}$ can be realized as a smooth function valued in $\calO_{\lambda}$ on certain Shimura set $\rmZ(\overline{\rmQ})$ associated to the group $\rmG(\overline{\rmQ})$ with suitable level structure. The Shimura set $\rmZ(\overline{\rmB})$ defined similarly from $\rmG(\overline{\rmB})$ can be embedded in $\rmZ(\overline{\rmQ})$ and this can be seen as an analogue of the Hirzebruch--Zagier morphism for the Shimura sets. We define the {\em distinguished period} by
 \begin{equation*}
 \calP_{\mathrm{dis}}(\phi^{\dagger})=\sum\limits_{z\in \rmZ(\overline{\rmB})}\phi^{\dagger}(z).
 \end{equation*}
Recall a cupsidal automorphic representation $\Pi$ of $\GL_{2}(\Adel_{\rmF})$ is called distinguished if there is a function $\phi$ in $\Pi$ such that
\begin{equation*}
\calP^{\GL_{2}}_{\mathrm{dis}}(\phi)=\int_{\GL_{2}(\mathbf{Q})\backslash\GL_{2}(\mathbf{A})}\phi(g)dg
\end{equation*}
is non-vanishing. It is known that $\Pi$ is distinguished only if $\Pi_{\infty}$ is in the discrete series and $\Pi$ comes from a cupsidal automorphic representation $\pi$ of $\GL_{2}(\Adel)$ via base change. The notion of distinguished representation has played an prominent role in the proof of the Tate conjecture of Hilbert--Blumenthal surfaces by \cite{HLR} and the quaternionic Hilbert--Blumenthal surfaces by \cite{Lai} and \cite{FH}. Moreover it is known that being distinguished is preserved under Jacquet--Langlands correspondence, see Proposition  
\ref{FH} for the precise statement. 

The above discussion should have explained that the distinguished period is of great arithmetic and representation theoretic interest and also why we stick to the situation when the period is defined by an automorphic form coming from base change. This set-up also leads us to consider a closely related period of $f$ namely the {Petersson norm} of $f^{\dagger}$, this is defined by
\begin{equation*}
\calP(f^{\dagger})=\sum\limits_{z\in \rmZ(\overline{\rmB})}f^{\dagger}(z)^{2}.
\end{equation*}
and is referred to as the {\em quaternionic period} of $f$ in \cite{Wang}. Our main result will be concerned with the $\lambda$-integrality of the ratio $\calP_{\mathrm{dis}}(\phi^{\dagger})/\calP(f^{\dagger})$ of these two periods. Before we state it, we introduce several assumptions that are imposed on the Galois representation attached to $f$ or equivalently to $f^{\dagger}$. Let $\rho_{\pi^{\circ}}: \rmG_{\QQ}\rightarrow \GL_{2}(\rmE_{\lambda})=\GL(\rmV_{\pi^{\circ}})$ be the Galois representation associated to $f^{\dagger}$ by the Eichler--Shimura construction and Jacquet--Langlands correspondence. Recall also that $\pi^{\circ}$ is the automorphic representation corresponding to $f^{\dagger}$. We denote by $\overline{\rho}_{\pi^{\circ}}$ the residual representation of $\rho_{\pi^{\circ}}$. Let $\Sigma^{+}$ be the set of primes dividing $\rmN^{+}$. Let $\Sigma^{-}_{\mathrm{ram}}$ be the set of $r$ dividing $\rmN^{-}$ such that $l\mid r^{2}-1$ and $\Sigma^{-}_{\mathrm{mix}}$ be the set of $r$ dividing $\rmN^{-}$ such that $l\nmid r^{2}-1$.

\begin{assumption}\label{intro-ass1}
We make the following assumptions on $\bar{\rho}_{\pi^{\circ}}$:
\begin{enumerate}
\item  $\bar{\rho}_{\pi^{\circ}}\vert_{\rmG_{\QQ(\zeta_{l})}}$ is absolutely irreducible;
\item The image of $\bar{\rho}_{\pi^{\circ}}$ contains $\GL_{2}(\FF_{l})$;
\item $\overline{\rho}_{\pi^{\circ}}$ is minimal at primes in $\Sigma^{+}$ in the sense that all the liftings of $\overline{\rho}_{\pi^{\circ}}\vert_{\rmG_{\QQ_{r}}}$ are minimally ramified for $r\in\Sigma^{+}$;
\item $\overline{\rho}_{\pi^{\circ}}$ is ramified at primes in $\Sigma^{-}_{\mathrm{ram}}$.
\end{enumerate}
\end{assumption}

Let $\rmT_{\pi^{\circ}}$ be a suitable Galois stable lattice in $\rho_{\pi^{\circ}}$ which we will make precise in the main body of this article and $\rmT_{\pi^{\circ}, n}$ be the reduction of $\rmT_{\pi^{\circ}}$ modulo $\lambda^{n}$. We will let $\rmM_{n}=\mathrm{Sym}^{2}(\rmT_{\pi^{\circ}, n})$ be the symmetric square of $\rmT_{\pi^{\circ}, n}$. 

\begin{theorem}\label{main-intro}
Let $f\in \rmS_{2}(\Gamma_{0}(\rmN))$ be a newform of weight $2$ with $\rmN=\rmN^{+}\rmN^{-}$ such that $\rmN^{-}$ is squarefree and has odd number of prime factors. Let $f^{\dagger}$ be the automorphic form on $\rmZ(\overline{\rmB})$ corresponding to $f$ under the Jacquet--Langlands correspondence and let $\phi^{\dagger}$ be the base change of $f^{\dagger}$ considered as an automorphic form on  $\rmZ(\overline{\rmQ})$. 
\begin{enumerate}
\item We assume that the residual Galois representation $\overline{\rho}_{\pi^{\circ}}$ satisfies Assumption \ref{intro-ass1}; 
\item We further assume that $\rmH^{1}(\Delta_{n}, \rmM_{n})=0$ 
for every $n\geq 1$ where $\Delta_{n}=\Gal(\QQ(\rmM_{n})/\QQ)$ for the splitting field $\QQ(\rmM_{n})$ of the Galois module $\rmM_{n}$. 
\end{enumerate}
Then we have the following inequality
\begin{equation*}
\ord_{\lambda}(\calP_{\mathrm{dis}}(\phi^{\dagger}))\geq \ord_{\lambda}(\calP(f^{\dagger})).
\end{equation*}
\end{theorem}

It should also be interesting to study the integrality question of the ratio between the distinguished period $\calP_{\mathrm{dis}}(f_{\rmF})$ for the base-change $f_{\rmF}$ of $f$ and the classical Petersson norm $\langle f, f\rangle$ of $f$. The $p$-adic reciprocity formuals proved in \cite{LLZ}, \cite{LSZ} should be more suitable for this purpose.  Our main result is clearly the compact version for this more classical ratio under the Jacquet--Langlands correspondence. We also suspect that when the $\lambda$-adic valuation of this ratio is positive, then the positive part may be related to the base-change congruence ideal as defined by Hida \cite{Hida}. We hope to return to this question in the near future. 

\subsection{Strategy of the proof} We now briefly outline the proof for the main integrality result in Theorem \ref{main-intro}. Let $\rmV_{\pi^{\circ}}$ be the representation space of $\rho_{\pi^{\circ}}$ and recall that we have fixed a lattice $\rmT_{\pi^{\circ}}$ of $\rmV_{\pi^{\circ}}$.  We define the divisible Galois module $\calM_{\pi^{\circ}}$ by
\begin{equation*}
0\rightarrow \rmT_{\pi^{\circ}}\rightarrow \rmV_{\pi^{\circ}}\rightarrow \calM_{\pi^{\circ}}\rightarrow 0.
\end{equation*}
The key step towards proving the integrality of the ratio of the distinguished period and the quaternionic period is to prove the following theorem bounding the Bloch--Kato Selmer group of $\mathrm{Ad}^{0}(\calM_{\pi^{\circ}})$ in terms of $\ord_{\lambda}(\calP_{\mathrm{dis}}(\phi^{\dagger}))$. 

\begin{theorem}\label{main-intro-1} 
Let $\nu=\mathrm{ord}_{\lambda}(\calP_{\mathrm{dis}}(\phi^{\dagger}))$ and $\eta=\varpi^{\nu}$. 
\begin{enumerate}
\item We assume that the residual Galois representation $\overline{\rho}_{\pi^{\circ}}$ is absolutely irreducible. 
\item We further assume that
\begin{equation*}
\rmH^{1}(\Delta_{n}, \rmM_{n})=0
\end{equation*} 
for every $n\geq 1$. 
\end{enumerate}
Then we have 
\begin{equation*}
\mathrm{length}_{\calO_{\lambda}}\phantom{.}\rmH^{1}_{f}(\QQ, \mathrm{Ad}^{0}(\calM_{\pi^{\circ}}))\leq \nu.
\end{equation*}
\end{theorem}

The proof of this theorem relies on the construction of the so-called Flach system of Jacquet--Langlands type on certain Shimura surface which we now describe in more details. Let $p$ be a prime which is \emph{inert} in $\rmF$.  We can obtain another quaternion algebra $\rmB$ by switching the local invariants of $\overline{\rmB}$ at $p$ and $\infty$. Then $\rmQ=\rmB\otimes\rmF$ is a totally indefinite quaternion algebra over $\rmF$ obtained from $\overline{\rmQ}$ by switching the invariants at the two archimedean places of $\rmF$. Note that $\pi^{\overline{\rmQ}}$ admits a Jacquet--Langlands transfer $\pi^{\rmQ}$ as a representation of $\rmG(\rmQ)$ which is the algebraic group given by the Weil restriction of $\rmQ^{\times}$ from $\rmF$ to $\QQ$. Let $\rho_{\pi^{\rmQ}}: \rmG_{\rmF}\rightarrow\GL_{2}(\rmE_{\lambda})= \GL( \rmV_{\pi^{\rmQ}})$ be the Galois representation of $\rmG_{\rmF}$ attached to $\pi^{\rmQ}$. Then it is well-known that the Asai representation $\mathrm{As}(\mathrm{\rho_{\pi^{\rmQ}}})=\mathrm{As}(\rmV_{\pi^{\rmQ}})$ can be realized on the middle degree cohomology of certain Shimura surface $\rmX(\rmQ)\otimes\QQ$ which we will refer to as the {\em quaternionic Hilbert--Blumenthal surface}. Roughly speaking these Shimura surfaces defines a moduli problem $\rmX(\rmQ)$ which classifies abelian fourfold with quaternionic multiplication by  $\rmQ$ and additional structures. The integral cohomology of $\rmX(\rmQ)\otimes\QQ$ with coefficient in $\calO_{\lambda}$ defines a natural lattice $\rmT_{\pi^{\rmQ}}$ in $\rmV_{\pi^{\rmQ}}$ and we define the divisible Galois module $\calM_{\pi^{\rmQ}}$ using the exact sequence
\begin{equation*}
0\rightarrow \rmT_{\pi^{\rmQ}}\rightarrow \rmV_{\pi^{\rmQ}}\rightarrow \calM_{\pi^{\rmQ}}\rightarrow0.
\end{equation*}
In our set-up, the Asai representation $\mathrm{As}(\calM_{\pi^{\rmQ}})(-1)$ splits into
\begin{equation*}
\mathrm{As}(\calM_{\pi^{\rmQ}})(-1)=\mathrm{Ad}^{0}(\calM_{\pi^{\circ}})\oplus\rmE_{\lambda}/\calO_{\lambda}(\omega_{\rmF/\QQ})
\end{equation*} 
where $\omega_{\rmF/\QQ}$ is the quadratic character associated to $\rmF$. Since $\As(\rmT_{\pi^{\rmQ}})$ is realized in the $\pi^{\rmQ}$-isotopic part $\rmH^{2}_{\pi^{\rmQ}}(\rmX(\rmQ)\otimes\overline{\QQ}, \calO_{\lambda}(2))$ of the middle degree cohomology of $\rmX(\rmQ)$ and $\mathrm{Ad}^{0}(\calM_{\pi^{\circ}})(1)$ is the Kummer dual of $\mathrm{Ad}^{0}(\calM_{\pi^{\circ}})$, we are naturally led to construct elements in the Galois cohomology group $\rmH^{1}(\QQ, \rmH^{2}_{\pi^{\rmQ}}(\rmX(\rmQ)\otimes\overline{\QQ}, \calO_{\lambda}(2)))$. For this purpose, we consider the motivic cohomology group $\rmH^{3}_{\calM}(\rmX(\rmQ)\otimes\QQ, \ZZ(2)))$ for the surface $\rmX(\rmQ)\otimes\QQ$ whose elements consist of pairs $(\rmZ, f)$ where $\rmZ$ is a curve on $\rmX(\rmQ)\otimes\QQ$ and $f$ is a rational function on $\rmX(\rmQ)\otimes\QQ$ with trivial Weil divisors. Let $\rmX(\rmB)\otimes\QQ$ be the Shimura curve associated to the indefinite quaternion algebra $\rmB$, then we have the quaternionic Hirzebruch--Zagier morphism
\begin{equation*}
\theta: \rmX(\rmB)\otimes\QQ\rightarrow \rmX(\rmQ)\otimes\QQ
\end{equation*}
which can even be defined integrally. Then we define the {\em Flach element} by the pair
\begin{equation*}
\Theta^{[p]}=(\theta_{\ast}\rmX(\rmB)\otimes\QQ, p).
\end{equation*}
The construction of this element is inspired by that of \cite{Flach} which is used to bound the symmetric square Selmer group for an elliptic curve. However, for our purpose, we do not need to include the Siegel modular unit in our definition (which is also not available for us as our cycle class is given by a compact Shimura curve). On the other hand, the original construction of the Flach element has applications to Iwasawa theory. In a closely related setting, this is carried out in \cite{LLZ}. Note that they consider the case when the Hilbert modular form is not arising from base change and hence our result is of complimentary nature to their work. There is an Abel--Jacobi map 
\begin{equation*}
\mathrm{AJ}_{\pi^{\rmQ}}:  \rmH^{3}_{\calM}(\rmX(\rmQ)\otimes\QQ, \ZZ(2))\rightarrow \rmH^{1}(\QQ, \rmH^{2}_{\pi^{\rmQ}}(\rmX(\rmQ)\otimes\overline{\QQ}, \calO_{\lambda}(2)))
\end{equation*}
defined using the Chern character map from the $K$-theory to \'etale cohomology. The class $\kappa^{[p]}$ defined by $\mathrm{AJ}_{\pi^{\rmQ}}(\Theta^{[p]})$ is called the {\em Flach class} in our setting. The local ramification behaviour of the class $\kappa^{[p]}$ will be analyzed. In particular at $p$, we are concerned with singular quotient 
\begin{equation*}
\partial_{p}(\kappa^{[p]})\in \rmH^{1}_{\sin}(\QQ_{p}, \rmH^{2}_{\pi^{\rmQ}}(\rmX(\rmQ)\otimes\overline{\QQ}_{p}, \calO_{\lambda}(2))) 
\end{equation*}
of $\kappa^{[p]}$. Note that $\rmH^{1}_{\sin}(\QQ_{p}, \rmH^{2}_{\pi^{\rmQ}}(\rmX(\rmQ)\otimes\overline{\QQ}_{p}, \calO_{\lambda}(2))$ is isomorphic to $\rmH^{2}_{\pi^{\rmQ}}(\rmX(\rmQ)\otimes\overline{\FF}_{p}, \calO_{\lambda}(1))^{\rmG_{\FF_{p}}}$. If we choose $p$ carefully, then $\rmH^{2}_{\pi^{\rmQ}}(\rmX(\rmQ)\otimes\overline{\FF}_{p}, \calO_{\lambda}(1))^{\rmG_{\FF_{p}}}$  is isomorphic to $\calO_{\lambda}[\rmZ(\overline{\rmQ})][\pi^{\overline{\rmQ}}]$ the $\pi^{\overline{\rmQ}}$-isotypic component of the space of the $\calO_{\lambda}$-valued functions on $\rmZ(\overline{\rmQ})$  by the Tate conjecture for $\rmX(\rmQ)\otimes\overline{\FF}_{p}$ proved in \cite{TX}. In this case, we can view $\partial_{p}(\kappa^{[p]})$ as an element of $\calO_{\lambda}[\rmZ(\overline{\rmQ})][\pi^{\overline{\rmQ}}]$. There is a natural pairing $(\cdot,\cdot)$ on this space $\calO_{\lambda}[\rmZ(\overline{\rmQ})]$. Then we can prove the following key {\em reciprocity formula}
\begin{equation*}
(\partial_{p}(\kappa^{[p]}), \phi^{\dagger})=\calP_{\mathrm{dis}}(\phi^{\dagger}).
\end{equation*}
Then an Euler system argument allows us to show that  $\eta=\varpi^{\ord_{\lambda}(\calP_{\mathrm{dis}}(\phi^{\dagger}))}$ annihilates the Selmer group $\rmH^{1}_{f}(\QQ, \mathrm{Ad}^{0}(\calM_{\pi^{\circ}}))$ under some technical assumptions. Using this annihilation result, it can be shown that $\mathrm{leng}_{\calO_{\lambda}}\phantom{.}\rmH^{1}_{f}(\QQ, \mathrm{Ad}^{0}(\calM_{\pi^{\circ}}))$ admits an upper-bound given by $\nu=\ord_{\lambda}(\calP_{\mathrm{dis}}(\phi^{\dagger}))$.

On the other hand, the Selmer group of the adjoint representation $\mathrm{Ad}^{0}(\calM_{\pi^{\circ}})$ is closely related to the deformation theory of the residual representation $\overline{\rho}_{\pi^{\circ}}$. In fact, as a by-product of the so-called $\rmR=\rmT$ theorem proved using the Taylor--Wiles method, it can be shown that $\rmH^{1}_{f}(\QQ, \mathrm{Ad}^{0}(\calM_{\pi^{\circ}}))$ has a subspace $\rmH^{1}_{\calS}(\QQ, \mathrm{Ad}^{0}(\calM_{\pi^{\circ}}))$ whose length is given by the $\lambda$-valuation of certain congruence number $\eta(\rmN^{+},\rmN^{-})$ that detects congruences between the modular form $f$ and other modular forms in $\rmS_{2}(\Gamma_{0}(\rmN))$ which are new at primes dividing $\rmN^{-}$. Finally, it is well--known that the $\lambda$-valuation of this congruence number is exactly that of the Petersson norm of $f^{\dagger}$ under our assumptions. This finally proves that $\ord_{\lambda}(\calP_{\mathrm{dis}}(\phi^{\dagger}))\geq \ord_{\lambda}(\calP(f^{\dagger}))$.

\subsection{Notations and conventions} We will use common notations and conventions in algebraic number theory and algebraic geometry. The cohomology theories appear in this article will be understood as the \'{e}tale cohomologies. For a field $\rmK$, we denote by $\rmK^{\ac}$ the separable closure of $\rmK$ and denote by $\rmG_{\rmK}:=\Gal(\rmK^{\ac}/\rmK)$ the Galois group of $\rmK$. For a finite place $v$ of $\rmK$, $\mathrm{Frob}_{v}$ is the arithmetic Frobenius at $v$.

We let $\Adel_{\rmF}$ be the ring of ad\`{e}les over a number field $\rmF$ and $\Adel^{\infty}_{\rmF}$ be the subring of finite ad\`{e}les.  For a finite place $v$ of $\rmF$, $\Adel^{\infty, v}_{\rmF}$ is the prime-to-$v$ part of  $\Adel^{\infty}_{\rmF}$.  When $\rmF=\QQ$, then we will omit the subscript $\rmF$ from these notations.

When $\rmK$ is a local field, we denote by $\calO_{\rmK}$ its valuation ring  and by $k$ its residue field. We let $\rmI_{\rmK}$ be the inertia subgroup of $\rmG_{\rmK}$. For a $\rmG_{\rmK}$-module $\rmM$, 
\begin{enumerate}
\item the finite part $\rmH^{1}_{\mathrm{fin}}(\rmK, \rmM)$ of $\rmH^{1}(\rmK, \rmM)$ is defined to be $\rmH^{1}(k, \rmM^{\rmI_{\rmK}})$;
\item the singular part $\rmH^{1}_{\mathrm{sin}}(\rmK, \rmM)$ of $\rmH^{1}(\rmK, \rmM)$ is defined to be the quotient of $\rmH^{1}(\rmK, \rmM)$ by the image of $\rmH^{1}_{\mathrm{fin}}(\rmK, \rmM)$ in $\rmH^{1}(\rmK, \rmM)$ via inflation;
\item let $x$ be an element of $\rmH^{1}(\rmK, \rmM)$, we call the image of $x$ in $\rmH^{1}_{\mathrm{sin}}(\rmK, \rmM)$ the singular residue of $x$ written as $\partial_{p}(x)$. 
\end{enumerate}
Let $\rmK$ be a number field and let $\rmK_{v}$ be the completion of $\rmK$ at a place $v$. Suppose $x\in\rmH^{1}(\rmK,\rmM)$, then we will write $\loc_{v}(x)$ the image of $x$ under the restriction map 
\begin{equation*}
\loc_{v}:\rmH^{1}(\rmK,\rmM)\rightarrow \rmH^{1}(\rmK_{v}, \rmM).
\end{equation*}
We will denote by $\partial_{v}(x)$ the image of $x$ in $ \rmH^{1}_{\mathrm{sin}}(\rmK_{v}, \rmM)$ under the composite map of
\begin{equation*}
\loc_{v}:\rmH^{1}(\rmK,\rmM)\rightarrow \rmH^{1}(\rmK_{v}, \rmM)
\end{equation*}
and the natural map $\rmH^{1}(\rmK_{v}, \rmM)\rightarrow \rmH^{1}_{\mathrm{sin}}(\rmK_{v}, \rmM)$.

\subsection*{Acknowledgements} We would like to thank Ming-Lun Hsieh for introducing the notion of distinguished representation to the author and pointing out useful references. We would like to thank Matteo Tamiozzo for useful email communications.

\section{Quaternionic Hilbert--Blumenthal surfaces}
\subsection{A quaternionic Shimura surface}
Let $\rmF$ be a real quadratic field of discriminant $\rmD_{\rmF}$. Let $\rmQ=\rmB\otimes \rmF$ for an indefinite quaternion algebra $\rmB$ with discriminant $\rmN^{-}p$ which is square free consisting of even number of prime factors. Suppose that $p$ is a prime inert in $\rmF$.  Given this $\rmQ$, let $\rmG(\rmQ)$ be the algebraic group over $\QQ$ defined by the Weil restriction of $\rmQ^{\times}$ from $\rmF$ to $\QQ$. Fix an open compact subgroup $\rmK$ of $\rmG(\rmQ)(\mathbf{A}^{\infty})$
we can associate to it a Shimura variety $\mathrm{Sh}_{\rmK}(\rmQ)$ over $\QQ$ whose complex points are given by
\begin{equation*}
\mathrm{Sh}_{\rmK}(\rmQ)(\CC)=\rmG(\rmQ)(\QQ)\backslash\calD\times \rmG(\rmQ)(\mathbf{A}^{\infty})/\rmK
\end{equation*}
where we denote by $\calD=\calH^{\pm 2}$ product of two copies of the upper and lower half planes.

Let $\widetilde{\rmG}(\rmQ)$ be the algebraic group  over $\QQ$ whose $\rmR$-points are given by
\begin{equation*}
\widetilde{\rmG}(\rmQ)(\rmR)=\{x\in (\rmQ\otimes \rmR)^{\times}: \rmN^{\circ}(x)\in\rmR^{\times}\}
\end{equation*}
for any $\QQ$-algebra $\rmR$. Fix an open compact subgroup $\widetilde{\rmK}$ of $\widetilde{\rmG}(\rmQ)(\mathbf{A}^{\infty})$,
we can associate to it a Shimura variety ${\mathrm{Sh}}_{\widetilde{\rmK}}(\rmQ)$ which represents the following moduli problem over $\QQ$. It associates to a scheme $\rmS\in \mathrm{Sch}/\QQ$ the isomorphism classes of triples $(\rmA, \iota, \overline{\eta})$ up to isogeny where
\begin{enumerate}
\item $\rmA$ is an abelian scheme over $\rmS$ of relative dimension $4$;
\item  $\iota: \rmB\otimes \rmF\rightarrow \End^{0}(\rmA)$ is a homomorphism such that
\begin{equation*}
\iota(b\otimes a)^{\ast}=\iota(b^{\ast}\otimes a)
\end{equation*}
where the first $(\cdot)^{\ast}$ means the Rosati involution of $\End^{0}(\rmA)=\End_{\rmS}(\rmA)\otimes \QQ$ while the second $(\cdot)^{\ast}$ means the main involution on $\rmB$;
\item $\overline{\eta}$ is a $\rmK$-equivalence class of $\rmB\otimes \rmF$-equivariant isomorphisms 
\begin{equation*}
\eta: \widehat{\rmV}(\rmA)\xrightarrow{\sim} \rmB\otimes\rmF(\mathbf{A}^{\infty})
\end{equation*} 
where $\widehat{\rmV}(\rmA)=\prod_{v}\rmT_{v}(\rmA)\otimes\QQ$ and which preserves Weil-pairing on the left-hand-side and the reduced trace pairing on the right-hand-side up to a scalar in $\mathbf{A}^{\infty\times}$.
\end{enumerate}
We will refer to \cite{Kottwitz} for the precise definitions of the notations used here. We require that the following Kottwitz conditions are satisfied 
\begin{equation*}
\det(\iota(b\otimes a);\mathrm{Lie}(\rmA))=\rmN^{\circ}(b)\rmN^{2}_{\rmF/\QQ}(a)
\end{equation*}
for $b\otimes a\in \rmB\otimes \rmF$. It is known that the polarization datum that are usually included in the above moduli problem can be omitted, see \cite[Lemma 3.8]{Zink}, \cite[Remark 2.5]{LT}. 
\begin{remark}
It is clear that the above moduli problem is the same as the $\GSpin$-type Shimura variety defined by Kulda--Raopoport in \cite[\S 1]{KR1} via Morita equivalence.
\end{remark}

For $\widetilde{\rmK}$ sufficiently small, this moduli problem is representable by a quasi-projective smooth scheme over $\QQ$. Note that when $\rmB\otimes\rmF=\rmM_{2}(\rmF)$, then we recover the classical Hilbert--Blumenthal surface. As before, the $\CC$-points of $\mathrm{Sh}_{\widetilde{\rmK}}(\rmQ)(\CC)$ can be described by the double coset space
\begin{equation*}
\mathrm{Sh}_{\widetilde{\rmK}}(\rmQ)(\CC)=\widetilde{\rmG}(\rmQ)(\QQ)\backslash\calD\times \widetilde{\rmG}(\rmQ)(\mathbf{A}^{\infty})/\widetilde{\rmK}.
\end{equation*}
When $\widetilde{\rmK}$ is the restriction of $\rmK$, there is a canonical map $\mathrm{Sh}_{\widetilde{\rmK}}(\rmQ)\rightarrow \mathrm{Sh}_{\rmK}(\rmQ)$ extending the natural map
\begin{equation*}
\widetilde{\rmG}(\rmQ)(\QQ)\backslash\calD\times \widetilde{\rmG}(\rmQ)(\mathbf{A}^{\infty})/\widetilde{\rmK}\rightarrow \rmG(\rmQ)(\QQ)\backslash\calD\times \rmG(\rmQ)(\mathbf{A}^{\infty})/\rmK
\end{equation*}
on $\CC$-points. The PEL-type Shimura surface $\mathrm{Sh}_{\widetilde{\rmK}}({\rmQ})$ will facilitate the study of the geometry of 
$\mathrm{Sh}_{\rmK}({\rmQ})$ whose cohomology carries the automorphic representation of interest. We will refer to $\mathrm{Sh}_{\rmK}({\rmQ})$ as the {\em quaternionic Hilbert--Blumenthal surface} associated to $\rmQ$ or simply as the \emph{quaternionic Shimura surface}. In the case when $\rmQ=\rmM_{2}(\rmF)$, $\mathrm{Sh}_{\rmK}({\rmQ})$ agrees with the classical Hilbert--Blumenthal surface. 

\subsection{Integral model for the quaternionnic Shimura surface} 
Let $\calO_{\rmB}\subset \rmB$ be a maximal order stable under the involution $\ast$ and we consider the order $\calO_{\rmQ}=\calO_{\rmB}\otimes \calO_{\rmF}$ in $\rmB\otimes\rmF$. Let $p$ be a prime which is inert in $\rmF$ and divides the discriminant of $\rmB$, $\widetilde{\rmK}^{p}\subset \widetilde{\rmG}(\rmQ)(\mathbf{A}^{\infty, p})$ be a compact open subgroup, we can define a moduli problem over $\ZZ_{(p)}$ which associates to each $\rmS\in\mathrm{Sch}/\ZZ_{(p)}$ the set of isomorphism classes of triples $(\rmA,\iota, \overline{\eta}^{p})$ where 
\begin{enumerate}
\item $\rmA$ is an abelian scheme over $\rmS$ of relative dimension $4$ up to prime to $p$-isogeny;
\item $\iota: \calO_{\rmB}\otimes\calO_{\rmF}\rightarrow \End(\rmA)\otimes\ZZ_{(p)}$ is a homomorphism such that
\begin{equation*}
\iota(b\otimes a)^{\ast}=\iota(b^{\ast}\otimes a)
\end{equation*}
where the first $(\cdot)^{\ast}$ means the Rosati involution of $\End_{\rmS}(\rmA)\otimes \ZZ_{(p)}$ while the second $(\cdot)^{\ast}$ means the main involution on $\rmB$.
\item $\overline{\eta}^{p}$ is a $\rmK^{p}$-equivalence class of $\calO_{\rmB}^{?}\otimes\calO_{\rmF}$-equivariant isomorphisms 
\begin{equation*}
\eta^{p}: \widehat{\rmV}^{p}(\rmA)\xrightarrow{\sim} \calO_{\rmB}\otimes\calO_{\rmF}(\mathbf{A}^{\infty,p})
\end{equation*} 
where $\widehat{\rmV}^{p}(\rmA)=\prod_{v\neq p}\rmT_{v}(\rmA)\otimes\QQ$ and $\eta^{p}$ preserves Weil-pairing on the left-hand-side and the reduced trace pairing on the right-hand-side up to a scalar in $\mathbf{A}^{\infty, p\times}$;
\end{enumerate}
We will refer to \cite{Kottwitz} for the precise definitions of the notations used here. We require that the following Kottwitz conditions are satisfied 
\begin{equation*}
\det(\iota(b\otimes a);\mathrm{Lie}(\rmA))=\rmN^{\circ}(b)\rmN^{2}_{\rmF/\QQ}(a)
\end{equation*}
for $b\otimes a\in \calO_{\rmB}\otimes \calO_{\rmF}$ which is understood as an identity of polynomial functions with coefficient in $\calO_{\rmS}$ as in \cite{Kottwitz}.  It follows from  \cite[\S 5]{Kottwitz} that the this moduli problem is representable by a smooth quasi-projective scheme $\widetilde{\rmX}_{\widetilde{\rmK}}(\rmQ)$ over $\ZZ_{(p)}$. Let $\widetilde{\rmK}=\widetilde{\rmK}_{p}\tilde{\rmK}^{p}$ where $\widetilde{\rmK}_{p}$ is the intersection of $(\calO_{\rmQ}\otimes\ZZ_{p})^{\times}$ with $\widetilde{\rmG}(\rmQ)(\QQ_{p})$. Then the generic fiber $\widetilde{\rmX}_{\widetilde{\rmK}}(\rmQ)\otimes \QQ$ is given by ${\mathrm{Sh}}_{\widetilde{\rmK}}(\rmQ)$ defined in the previous subsection. The canonical map ${\mathrm{Sh}}_{\widetilde{\rmK}}(\rmQ)\rightarrow \mathrm{Sh}_{\rmK}(\rmQ)$ extends to the integral model $\rmX_{\rmK}(\rmQ)$ of $\mathrm{Sh}_{\rmK}(\rmQ)$ and gives rise to a finite map 
\begin{equation}\label{finite-map}
\widetilde{\rmX}_{\widetilde{\rmK}}(\rmQ)\longrightarrow {\rmX}_{\rmK}(\rmQ)
\end{equation}
between schemes over $\ZZ_{(p)}$. 

We will now fix a definite choice of the level structure for later purpose. We will do this by introducing a slight variant of the above moduli problem. Let $\rmN^{+}$ be an integer that coprime to $\rmN^{-}p$ and we put $\rmN=\rmN^{+}\rmN^{-}$. Let $d\geq 4$ be another integer. Then we consider the moduli problem over $\ZZ[1/\rmN\rmD_{\rmF}]$ that assigns for each $\rmS\in\mathrm{Sch}/\ZZ[1/\rmN\rmD_{\rmF}]$ the set of isomorphism classes of $4$-tuples $(\rmA, \iota, \rmC_{\rmN^{+}}, \alpha_{d})$ where
\begin{enumerate}
\item $\rmA$ is an abelian scheme over $\rmS$ of relative dimension $4$;
\item $\iota: \calO_{\rmB}\otimes\calO_{\rmF}\rightarrow \End(\rmA)$ is a homomorphism such that
\begin{equation*}
\iota(b\otimes a)^{\ast}=\iota(b^{\ast}\otimes a)
\end{equation*}
where the first $(\cdot)^{\ast}$ means the Rosati involution of $\End_{\rmS}(\rmA)$ while the second $(\cdot)^{\ast}$ means the main involution on $\rmB$;
\item $\rmC_{\rmN^{+}}$ is an $\calO_{\rmB}\otimes\calO_{\rmF}$-stable flat subgroup of $\rmA[\frakn^{+}]$ for $\frakn^{+}=\rmN^{+}\calO_{\rmF}$ such that at every geometric point, it is isomorphic to $(\calO_{\rmF}/\frakn^{+})^{2}$;
\item $\alpha_{d}: (\calO_{\rmF}/d)_{\rmS}^{2}\hookrightarrow \rmA[\frakd]$ is an $\calO_{\rmB}\otimes\calO_{\rmF}$-equivariant injection of group schemes over $\rmS$ for the ideal $\frakd=d\calO_{\rmF}$.
\end{enumerate}
This moduli problem is representable by a scheme $\widetilde{\rmX}_{\rmN^{+}, d}(\rmQ)$. The corresponding integral model for the Shimura variety associated to the group $\rmG(\rmQ)$ will be denoted by $\rmX_{\rmN^{+}, d}(\rmQ)$  and there is a finite morphism
\begin{equation*}
\widetilde{\rmX}_{\rmN^{+}, d}(\rmQ)\rightarrow \rmX_{\rmN^{+}, d}(\rmQ)
\end{equation*}
as in \eqref{finite-map}. Since we will not need the precise form of the level structure, we will refer the reader to \cite[Example 2.12]{LT} for the precise description of the open compact subgroup $\rmK_{\rmN^{+}, d}$ defining the level structure. When there is no danger of confusion, we will simply write $\rmX(\rmQ)$ for the scheme $\rmX_{\rmN^{+}, d}(\rmQ)$.

\subsection{Supersingular locus of the quaternionic Shimura surface}
Let $\FF$ be a fixed algebraic closure of $\FF_{p}$. We set $\rmD=\rmB\otimes \QQ_{p}$ to be the quaternionic division algebra over $\QQ_{p}$. Let $\calL$ be an isocrystal of height $4$ over $\FF$ with an action $\iota: \rmD\otimes\QQ_{p^{2}}\rightarrow \End(\calL)$ of $\rmD\otimes\QQ_{p^{2}}$ on $\calL$. Since we always have $\rmD\otimes\QQ_{p^{2}}=\rmM_{2}(\QQ_{p^{2}})$, we have $\calL=\calN^{2}$ and $\calN$ is equipped with an action of $\QQ_{p^{2}}$. %The polarization $\lambda$ gives rise to a non-degenerate alternating form  $(\cdot,\cdot): \calN\times \calN \rightarrow \rmK_{0}$ with the property that $(\rmF x, y)= (x, \rmV y)^{\sigma}$ and $(bx, y)= (x, by)$ for $b\in\QQ_{p^{2}}$. 
We will use covariant Dieudonn\'{e} theory throughout this article. A Dieudonn\'{e} lattice $\rmM$ is a lattice in $\calN$  with the property that $p\rmM \subset \rmF\rmM\subset \rmM$.  A Dieudonn\'{e} lattice is superspecial if $\rmF^{2}\rmM=p\rmM$ and in this case $\rmF=\rmV$. We are concerned with a Dieudonn\'{e} lattice $\rmM$ with an additional endomorphism $\iota: \ZZ_{p^{2}}\rightarrow \End(\rmM)$. We can decompose $\rmM$ as $\rmM_{0}\oplus \rmM_{1}$ according to the action of $\ZZ_{p^{2}}$. 

Let $\Nilp$ be the category of $\rmW_{0}$-schemes over which $p$ is locally nilpotent. We consider the set valued functor $\calM$ that sends $\rmS\in\Nilp$ to the isomorphism classes of the collection $(\rmX, \iota_{\rmX},  \lambda_{\rmX}, \rho_{\rmX})$ where:
\begin{enumerate} 
\item $\mathrm{X}$ is a $p$-divisible group of dimension $2$ and height $4$ over $\rmS$;
\item $\iota_{\rmX}:\ZZ_{p^{2}}\rightarrow \End(\rmX)$ is an action of $\ZZ_{p^{2}}$ on $\rmX$ defined over $\rmS$;
\item $\rho_{\rmX}: \rmX\times_{\rmS} \overline{\rmS}\rightarrow \XX\times_{\FF}\overline{\rmS}$ is an $\ZZ_{p^{2}}$-linear quasi-isogeny over $\overline{\rmS}$ which is the special fiber of $\rmS$ at $p$.
\end{enumerate}

We require that $\iota_{\rmX}$ satisfies the Kottwitz condition 
\begin{equation*}
\det(\iota_{\rmX}(a); \Lie(\rmX))=\rmN_{\QQ_{p^{2}}/\QQ_{p}}(a)
\end{equation*}
for $a\in \ZZ_{p^{2}}$. For $\rho_{\rmX}: \rmX\times_{\rmS} \overline{\rmS}\rightarrow \XX\times_{\FF}\overline{\rmS}$, we require that $\rho_{\rmX}^{-1}\circ\lambda_{\XX}\circ\rho_{\rmX}=c(\rho_{\rmX})\lambda_{\rmX}$
for a $\QQ_{p}$-multiple $c(\rho_{\rmX})$. This moduli problem is representable by a formal scheme $\calM$, locally formally of finite type over $\Spf{\mathrm{W}_{0}}$. We will be mainly concerned with the underlying reduced closed subscheme $\calM_{\FF}$ of $\calM$. Let $x=(\rmX, \iota_{\rmX},  \lambda_{\rmX}, \rho_{\rmX})$ be a point in $\calM(\FF)$ and let $\rmM$ be the Dieudonn\'e lattice of $\rmX$. The action of $\ZZ_{p^{2}}$-action on $\rmX$ gives rise to a grading of the Dieudonn\'e module $\rmM=\rmM_{0}\oplus\rmM_{1}$. For this lattice, we always have
\begin{equation*}
\begin{aligned}
&p\rmM_{0}\subset^{1}\rmV\rmM_{1}\subset^{1}\rmM_{0}\\
&p\rmM_{1}\subset^{1}\rmV\rmM_{0}\subset^{1}\rmM_{1}.\\
\end{aligned}
\end{equation*}
We say $i\in\{0,1\}$ is a {\em critical index} for $\rmM$ with respect to the $\ZZ_{p^{2}}$-action if $\rmV^{2}\rmM_{i}=p\rmM_{i}$ and in this case we say $i$ is a critical index of the point $x$. It is clear that $\rmM$ is superspecial if $0$ and $1$ are both critical indices of $\rmM$ and in this case we say $x$ is a superspecial point. 

\begin{lemma}\label{RZ-space}
We have the following statements:
\begin{enumerate}
\item for any Dieudonn\'e lattice $\rmM$ associated to a point in $\calM(\FF)$, at least one $i\in\{0,1\}$ is critical for $\rmM$;
\item we have a partition of the scheme $\calM_{\FF}=\calM^{\circ}_{\FF}\cup\calM^{\bullet}_{\FF}$ where $\calM^{\circ}_{\FF}$ consists those points of $\calM_{\FF}$ such that $i=0$ is a critical index of the associated Dieudonn\'e lattice and $\calM^{\bullet}_{\FF}$ consists those points of $\calM_{\FF}$ such that $i=1$ is a critical index of the associated Dieudonn\'e lattice;
\item the irreducible components of $\calM^{\circ}_{\FF}$ and $\calM^{\bullet}_{\FF}$ are projective lines. These two family of projective lines intersect at the superspecial points of $\calM_{\FF}$.
\end{enumerate}
\end{lemma}
\begin{proof}
The first part $(1)$ follows from \cite[Lemma 4.2]{KR1}. The second part $(2)$ follows from $(1)$.  The third part $(3)$ is \cite[Proposition 4.4]{KR1} but we will recall briefly the construction of those projective lines: suppose $\rmM$ is associated to a point in $\calM^{\circ}_{\FF}$, then $\Lambda_{0}=\rmM^{p\rmV^{-2}}_{0}$ is a lattice over $\ZZ_{p^{2}}$, then we associate the projective line given by $\PP(\Lambda_{0}/p\Lambda_{0})$. Suppose $\rmM$ is associated to a point in $\calM^{\bullet}_{\FF}$, then $\Lambda_{1}=\rmM^{p\rmV^{-2}}_{1}$ is a lattice over $\ZZ_{p^{2}}$, then we associate the projective line given by $\PP(\Lambda_{1}/p\Lambda_{1})$. It is clear these projective lines are defined over $\FF_{p^{2}}$.
\end{proof}
\begin{remark}
We will refer to those projective lines in $\calM^{\circ}$ as projective lines of $\circ$-type and those projective lines in $\calM^{\bullet}$ as projective lines of $\bullet$-type.
\end{remark}
Let $\overline{\rmQ}$ be the quaternion algebra obtained from $\rmQ$ by switching the invariants at the archimedean places of $\rmF$, then we define the algebraic group $\rmG(\overline{\rmQ})$ over $\QQ$ given by the Weil restriction of $\overline{\rmQ}^{\times}$ from $\rmF$ to $\QQ$. Let $\rmK$ be an open compact subgroup of $\rmG(\rmQ)(\mathbf{A}^{\infty})$, then $\rmK$ can be viewed as an open compact subgroup of $\rmG(\overline{\rmQ})(\mathbf{A}^{\infty})$. Let $\overline{\rmX}(\rmQ)$ be the special fiber of the quaternionic surface $\rmX(\rmQ)$ over $\FF$.
The Rapoport--Zink uniformization theorem \cite[Theorem 6.1]{RZ}, \cite[Theorem 1.2]{Shen} furnishes the following description of the supersingular locus $\overline{\rmX}^{\ss}(\rmQ)$ of $\overline{\rmX}(\rmQ)$.

\begin{proposition}\label{ss-locus}
The supersingular locus $\overline{\rmX}^{\ss}(\rmQ)$ of $\overline{\rmX}(\rmQ)$ is pure of dimension $1$.
\begin{enumerate}
\item We have an isomorphism from the double quotient 
\begin{equation*}
\rmG(\overline{\rmQ})(\QQ)\backslash\calM_{\FF}\times \rmG(\overline{\rmQ})(\mathbf{A}^{\infty,p})/\rmK^{p}
\end{equation*}
and $\overline{\rmX}^{\ss}(\rmQ)$ which descends to an isomorphism over $\FF_{p^{2}}$;
\item The irreducible components of $\overline{\rmX}^{\ss}(\rmQ)$ are projective lines which are parametrized by two copies of 
\begin{equation*}
\rmZ(\overline{\rmQ})=\rmG(\overline{\rmQ})(\QQ)\backslash\rmG(\overline{\rmQ})(\mathbf{A}^{\infty})/\rmK^{p}\rmK_{p}
\end{equation*}
where $\rmK_{p}$ is  $(\calO_{\overline{\rmQ}}\otimes\ZZ_{p})^{\times}$. We denote these two copies by $\rmZ^{\circ}(\overline{\rmQ})$ and $\rmZ^{\bullet}(\overline{\rmQ})$ respectively, they parametrize those projective lines of $\circ$-type and those projective lines of $\bullet$-type respectively;
\item The superspecial locus $\overline{\rmX}^{\ssp}(\rmQ)$ is a discrete set of points parametrized by 
\begin{equation*}
\rmZ_{\Iw}(\overline{\rmQ})=\rmG(\overline{\rmQ})(\QQ)\backslash\rmG(\overline{\rmQ})(\mathbf{A}^{\infty})/\rmK^{p}\Iw_{p}
\end{equation*}
where $\Iw_{p}$ is the Iwahori subgroup of $\rmK_{p}$.
\end{enumerate}
\end{proposition}
\begin{proof}
The first part $(1)$ follows from the Rapoport--Zink uniformization theorem and the fact that the Weil descent datum for $\calM$ is effective over $\ZZ_{p^{2}}$. The second part $(2)$ and the third part $(3)$ follow from the first part and descriptions in Lemma \ref{RZ-space} immediately.
\end{proof}

\begin{remark}
By the previous proposition we can write $\overline{\rmX}^{\ss}(\rmQ)=\overline{\rmX}^{\circ}(\rmQ)\cup \overline{\rmX}^{\bullet}(\rmQ)$ where
\begin{enumerate}
\item $\overline{\rmX}^{\circ}(\rmQ)$ is a $\PP^{1}$-bundle over $\rmZ^{\circ}(\rmQ)$ which we will call the $\circ$-component for $\overline{\rmX}(\rmQ)$; 
\item $\overline{\rmX}^{\bullet}(\rmQ)$ is a $\PP^{1}$-bundle over $\rmZ^{\bullet}(\rmQ)$ which we will call the $\bullet$-component for $\overline{\rmX}(\rmQ)$.
\end{enumerate}
\end{remark}

\begin{remark}
Note that the description from the previous Proposition \ref{ss-locus} can also be obtained using the so-called isogeny trick instead of using Rapoport--Zink uniformization theorem, see the main result of \cite{TX} which treats much more general quaternionic Shimura varieties.
\end{remark}

\section{Quaternionic Hirzebruch--Zagier divisor}
\subsection{Shimura curves and Drinfeld uniformization}
Let $\rmB$ be the indefinite quaternion algebra over $\QQ$ with discriminant $p\rmN^{-}$ considered in the last section. We denote by $\rmG(\rmB)$ the algebraic group over $\QQ$ given by $\rmB^{\times}$. We fix an open compact subgroup $\rmK$ of $\rmG(\rmB)(\mathbf{A}^{\infty})$ such that $\rmK^{p}\rmK_{p}$ where $\rmK^{p}$ is sufficiently small and $\rmK_{p}=(\calO_{\rmB}\otimes\ZZ_{p})^{\times}$. 
We consider the moduli problem over $\ZZ_{(p)}$ which assigns each $\rmS\in \mathrm{Sch}/\ZZ_{(p)}$ the set of triples $(\rmA, \iota, \overline{\eta})$ where
\begin{enumerate}
\item $\rmA$ is an abelian scheme over $\rmS$ of relative dimension $2$, up to prime to $p$-isogeny;
\item $\iota: \calO_{\rmB}\hookrightarrow \End_{\rmS}(\rmA)$ is an $\calO_{\rmB}$-action on $\rmA$ which is special in the sense of \cite[131--132]{BC};
\item $\overline{\eta}^{p}$ is a $\rmK^{p}$-equivalence class of $\calO_{\rmB}$-equivariant isomorphisms 
\begin{equation*}
\eta^{p}: \widehat{\rmV}^{p}(\rmA)\xrightarrow{\sim} \calO_{\rmB}\otimes\mathbf{A}^{\infty, p}
\end{equation*} 
where $\widehat{\rmV}^{p}(\rmA)=\prod_{v\neq p}\rmT_{v}(\rmA)\otimes\QQ$ and $\eta^{p}$ preserves Weil-pairing on the left-hand-side and the reduced trace pairing on the right-hand-side up to a scalar in $(\mathbf{A}^{\infty,p})^{\times}$. 
\end{enumerate}

It is well known this moduli problem is representable by a projective scheme over $\ZZ_{(p)}$ of relative dimension $1$ denoted by $\rmX_{\rmK}(\rmB)$. Its generic fiber $\mathrm{Sh}_{\rmK}(\rmB)=\rmX_{\rmK}(\rmB)_{\QQ}$ is the canonical model of the Shimura variety determined by $\rmB$ and $\rmK$ over $\QQ$. The set of $\CC$-points of $\mathrm{Sh}_{\rmK}(\rmB)$ is given by the following double coset space
\begin{equation*}
\mathrm{Sh}_{\rmK}(\rmB)(\CC)=\rmG(\rmB)(\QQ)\backslash \calH^{\pm} \times \rmG(\rmB)(\mathbf{A}^{\infty})/\rmK.
\end{equation*}

We will now fix a definite choice of the level structure as we did for the quaternionic Hilbert--Blumenthal surface. We introduce a moduli problem over $\ZZ[1/\rmN d]$  that assigns each $\rmS\in\mathrm{Sch}/\ZZ[1/\rmN d]$  the set of isomorphism classes of $4$-tuples $(\rmA, \iota, \rmC_{\rmN^{+}}, \alpha_{d})$ where
\begin{enumerate}
\item $\rmA$ is an abelian scheme over $\rmS$ of relative dimension $2$;
\item $\iota: \calO_{\rmB}\rightarrow \End_{\rmS}(\rmA)$ is a homomorphism which is special in the sense of \cite[131--132]{BC}; 
\item $\rmC_{\rmN^{+}}$ is an $\calO_{\rmB}$-stable finite flat subgroup of $\rmA[\rmN^{+}]$ such that at every geometric point, it is cyclic of order $(\rmN^{+})^{2}$;
\item $\alpha_{d}: (\ZZ/d\ZZ)^{2}_{\rmS}\hookrightarrow \rmA[d]$ is an $\calO_{\rmB}$-equivariant injection of group schemes over $\rmS$.\end{enumerate}
This moduli problem is representable by a scheme ${\rmX}_{\rmN^{+}, d}(\rmB)$ once $d\geq4$. When there is no danger of confusion, we will denote this scheme simply by $\rmX(\rmB)$. We will refer the readers to \cite[\S 2.1]{Wang} for the precise definition of the open compact subgroup corresponding to the level structure in the above moduli problem.  

Let $\FF$ be an algebraic closure of $\FF_{p}$ and $\rmW_{0}=\rmW(\FF)$ be the ring of Witt vectors of $\FF$ whose fraction field is denoted by $\rmK_{0}$. Let $\rmD$ be the quaternion division algebra over $\QQ_{p}$ and $\calO_{\rmD}$ be the maximal order of $\rmD$. We will present $\calO_{\rmD}$ as
$\calO_{\rmD}=\ZZ_{p^{2}}[\Pi]/(\Pi^{2}-p)$
with $\Pi a=\sigma(a)\Pi$ for $a\in \ZZ_{p^{2}}$. 
We first recall the notion of a \emph{special formal $\calO_{D}$-module}. Let $\rmS$ be a $\rmW_{0}$-scheme. A special formal $\calO_{\rmD}$-module over $\rmS$ is a formal $p$-divisible group $\rmX$ of dimension $2$ and height $4$ with an $\calO_{\rmD}$-action 
$\iota: \calO_{\rmD}\rightarrow \End_{\rmS}(\rmX)$ such that $\Lie(\rmX)$ is a locally free $\ZZ_{p^{2}}\otimes \calO_{\rmS}$-module of rank $1$. We fix a special $\calO_{\rmD}$-module $\mathbb{X}$ over $\FF$ whose Dieudonn\'{e} module is denoted by $\mathbb{M}$. 

Let $\calM_{\mathrm{Dr}}$ be the set valued functor that sends $S\in (\mathrm{Nilp})$ to the set of  isomorphism classes of pairs $(\rmX, \rho_{\rmX})$ where
\begin{enumerate}
\item $\rmX$ is special formal $\calO_{\rmD}$-module;
\item $\rho_{\rmX}: \rmX\times_{\rmS}\overline{\rmS}\rightarrow \mathbb{X}\times_{\FF} \overline{\rmS}$ is a quasi-isogney. 
\end{enumerate}
The functor $\calM_{\mathrm{Dr}}$ is represented by a formal scheme over $\rmW_{0}$ which we also denote by $\calM_{\Dr}$. The formal scheme $\calM_{\mathrm{Dr}}$ decomposes into a disjoint union
\begin{equation*}
\calM_{\mathrm{Dr}}=\bigsqcup_{i\in\ZZ}\calM_{\mathrm{Dr}, i}
\end{equation*}
according to the height $i$ of the quasi-isogeny $\rho_{\rmX}$. Each formal scheme $\calM_{\mathrm{Dr}, i}$ is isomorphic to the \emph{$p$-adic upper half plane} $\calH_{p}$. The group $\GL_{2}(\QQ_{p})$ acts naturally on the formal scheme $\calM_{\mathrm{Dr}}$ and each $\calM_{\Dr, i}$ affords an action of the group 
\begin{equation*}
\GL^{0}_{2}(\QQ_{p}):=\{g\in\GL_{2}(\QQ_{p}): \ord_{p}(\det(g))=0\}. 
\end{equation*}
Let $(\rmX, \rho)\in \calM_{\Dr}(\FF)$ and $\rmM$ be the Dieudonn\'{e} lattice of $\rmX$. The action of $\ZZ_{p^{2}}$ on $\rmX$ induces a grading $\rmM=\rmM_{0}\oplus \rmM_{1}$ that satisfies
\begin{equation}
\begin{aligned}
& p\rmM_{0}\subset^{1} \rmV\rmM_{1}\subset^{1} \rmM_{0}\hphantom{aa}p\rmM_{1}\subset^{1} \rmV\rmM_{0}\subset^{1} \rmM_{1};\\
& p\rmM_{0}\subset^{1} \Pi \rmM_{1}\subset^{1} \rmM_{0}\hphantom{aa}p\rmM_{1}\subset^{1} \Pi \rmM_{0}\subset^{1} \rmM_{1}.\\
\end{aligned}
\end{equation}
Since the action of $\Pi$ and $\rmV$ commute, we have the induced maps 
\begin{equation}
\begin{aligned}
& \Pi: \rmM_{0}/\rmV\rmM_{1}\rightarrow \rmM_{1}/\rmV\rmM_{0},\\  
& \Pi: \rmM_{1}/\rmV\rmM_{0}\rightarrow \rmM_{0}/\rmV\rmM_{1}.\\
\end{aligned}
\end{equation}
Since both $\rmM_{0}/\rmV\rmM_{1}$ and $\rmM_{1}/\rmV\rmM_{0}$ are one dimensional and the composite of the two maps is obviously zero, we can conclude that there is an $i\in \{0, 1\}$ such that $\Pi \rmM_{i}\subset \rmV\rmM_{i}$. Since both $\Pi \rmM_{i}$ and $\rmV\rmM_{i}$ are of colength $1$ in $\rmM_{i+1}$, we conclude that they are in fact equal to each other. We say that $i$ is a \emph{critical index} for $\rmM$ with respect to the $\rmD$-action if $\rmV\rmM_{i}=\Pi \rmM_{i}$ and a critical index always exists for $\rmM$. Let $\tau=\Pi^{-1}\rmV$ and it acts as an automorphism on $\rmM_{i}$ if $i$ is a critical index. If $0$ is a critical index, then we set $\Lambda_{0}=\rmM^{\tau=1}_{0}$ and this is a $\ZZ_{p}$-lattice of rank $2$ and we associate to it the projective line $\PP(\Lambda_{0}/p\Lambda_{0})$. Then $\rmV\rmM_{1}/p\rmM_{0}\subset^{1} \rmM_{0}/p\rmM_{0}=\Lambda_{0}/p\Lambda_{0}\otimes \FF$ gives a point on $\PP(\Lambda_{0}/p\Lambda_{0})(\FF)$. If $1$ is a critical index, then we similarly put $\Lambda_{1}=\Pi \rmM^{\tau=1}_{1}$ and we again associate to it the projective line $\PP(\Lambda_{1}/p\Lambda_{1})$. Similarly $\rmV\rmM_{0}/p\rmM_{1}\subset^{1} \rmM_{1}/p\rmM_{1}=\Lambda_{1}/p\Lambda_{1}\otimes \FF$ gives a point on $\PP(\Lambda_{1}/p\Lambda_{1})(\FF)$. This discussion leads to the following lemma. 

\begin{lemma}\label{Dr-space}
The space $\calM_{\mathrm{Dr},\FF}$ admits the following descriptions.
\begin{enumerate}
\item For any Dieudonn\'e lattice $\rmM$ associated to a point in $\calM_{\Dr}(\FF)$, at least one $i\in\{0,1\}$ is critical;
\item We have a partition of the scheme $\calM_{\Dr, \FF}=\calM^{\circ}_{\Dr, \FF}\cup\calM^{\bullet}_{\Dr, \FF}$ where $\calM^{\circ}_{\Dr,\FF}$ consists those points of $\calM_{\Dr, \FF}$ such that $i=0$ is a critical index and $\calM^{\bullet}_{\Dr, \FF}$ consists those points of $\calM_{\Dr, \FF}$ such that $i=1$ is a critical index;
\item The irreducible components of $\calM^{\circ}_{\Dr, \FF}$ and $\calM^{\bullet}_{\Dr, \FF}$ are projective lines. These two family of projective lines intersect at the superspecial points of $\calM_{\Dr, \FF}$.
\end{enumerate}
\end{lemma}

The Cerednick--Drinfeld uniformization theorem provides the following proposition describing the special fiber $\overline{\rmX}(\rmB)_{\FF}$ of $\rmX(\rmB)$. Let $\overline{\rmB}$ be the definite quaternion algebra obtained from $\rmB$ by switching the invariant at $p$ and $\infty$. Then we define $\rmG(\overline{\rmB})$ to be algebraic group defined by $\overline{\rmB}^{\times}$. Let $\rmK$ be an open compact subgroup of $\rmG(\rmB)(\mathbf{A}^{\infty})$, then $\rmK^{p}$ can be viewed naturally as an open compact of $\rmG(\overline{\rmB})(\mathbf{A}^{\infty, p})$.

\begin{proposition}\label{curve-red}
We have the following descriptions of the scheme $\overline{\rmX}(\rmB)$.
\begin{enumerate}
\item We have an isomorphism from the double quotient 
\begin{equation*}
\rmG(\overline{\rmB})(\QQ)\backslash\calM_{\Dr, \FF}\times \rmG(\overline{\rmB})(\mathbf{A}^{\infty, p})/\rmK^{p}
\end{equation*}
to $\overline{\rmX}(\rmB)_{\FF}$ which descends to an isomorphism over $\FF_{p^{2}}$;
\item The irreducible components of $\overline{\rmX}(\rmB)$ are projective lines which are parametrized by two copies of the Shimura set
\begin{equation*}
\rmZ(\overline{\rmB})=\rmG(\overline{\rmB})(\QQ)\backslash\rmG(\overline{\rmB})(\mathbf{A}^{\infty, p})/\rmK^{p}\overline{\rmK}_{p}
\end{equation*}
where $\overline{\rmK}_{p}$ is the group $(\calO_{\overline{\rmB}}\otimes\ZZ_{p})^{\times}$. We denote these two copies by $\rmZ^{\circ}(\overline{\rmB})$ and $\rmZ^{\bullet}(\overline{\rmB})$, they parametrize those projective lines corresponding to critical index $i=0$ resp. critical index $i=1$;

\item The superspecial locus $\overline{\rmX}(\rmB)$ is a discrete set of points parametrized by 
\begin{equation*}
\rmZ_{\Iw}(\overline{\rmB})=\rmG(\overline{\rmB})(\QQ)\backslash\rmG(\overline{\rmB})(\mathbf{A}^{\infty, p})/\rmK^{p}\Iw_{p}
\end{equation*}
where $\Iw_{p}$ is the Iwahori subgroup of $\overline{\rmK}_{p}$.
\end{enumerate}
\end{proposition}
\begin{proof}
This is a well-known result following from the Cerednick--Drinfeld uniformization theorem for the Shimura curve $\rmX(\rmB)$. For example, see the proof of \cite[Proposition 3.2]{Wang} for more details.
\end{proof}

\begin{remark}
By the previous proposition we can write $\overline{\rmX}(\rmB)=\overline{\rmX}^{\circ}(\rmB)\cup \overline{\rmX}^{\bullet}(\rmB)$ where 
\begin{enumerate}
\item $\overline{\rmX}^{\circ}(\rmB)$ is a $\PP^{1}$-bundle over $\rmZ^{\circ}(\overline{\rmB})$ which we will call the $\circ$-component;  
\item $\overline{\rmX}^{\bullet}(\rmB)$ is a $\PP^{1}$-bundle over $\rmZ^{\bullet}(\overline{\rmB})$ which we will call the $\bullet$-component.
\end{enumerate}
\end{remark}

\subsection{Quaternionic Hirzebruch--Zagier divisor} 
Let $(\rmA, \iota, \rmC_{\rmN^{+}}, \alpha_{d})$ be an element of $\rmX(\rmB)(\rmS)$ for a scheme $\rmS$ over $\ZZ[1/\rmN d\rmD_{\rmF}]$, we define 
\begin{enumerate}
\item $\widetilde{\rmA}=\rmA\otimes \calO_{\rmF}$ is the abelian scheme given by the Serre's tensor construction from $\rmA$;
\item $\widetilde{\iota}: \calO_{\rmB}\otimes\calO_{\rmF}\hookrightarrow \End (\rmA)\otimes\calO_{\rmF}\hookrightarrow \End (\widetilde{\rmA})$ be the  morphism given by $\iota\otimes\calO_{\rmF}$;
\item $\widetilde{\rmC}_{\rmN^{+}}\subset \widetilde{\rmA}[\frakn^{+}]$ be the $\calO_{\rmB}\otimes\calO_{\rmF}$-stable finite flat subgroup of $\widetilde{\rmA}$ given by 
\begin{equation*}
\widetilde{\rmC}_{\rmN^{+}}=\rmC_{\rmN^{+}}\otimes\calO_{\rmF};
\end{equation*}
\item $\widetilde{\alpha}_{d}: (\calO_{\rmF}/d)^{2}\hookrightarrow \rmA[\frakd]$ is the $\calO_{\rmB}\otimes\calO_{\rmF}$-equivariant injection
\begin{equation*}
\widetilde{\alpha}_{d}: (\ZZ/d)^{2}\otimes\calO_{\rmF}\hookrightarrow \rmA\otimes\calO_{\rmF}[d]=\widetilde{\rmA}[\frakd].
\end{equation*}
given by $\alpha_{d}\otimes\calO_{\rmF}$.
\end{enumerate}
This defines an element of $(\widetilde{\rmA}, \widetilde{\iota}, \widetilde{\rmC}_{\rmN^{+}}, \widetilde{\alpha}_{d})\in\widetilde{\rmX}(\rmQ)(\rmS)$ over $\ZZ[1/\rmN d\rmD_{\rmF}]$ which gives a map from $\rmX(\rmB)\otimes\ZZ[1/\rmN d \rmD_{\rmF}]$ to $\widetilde{\rmX}(\rmQ)$.  We therefore obtain the {\em quaternionic Hirzebruch--Zagier morphism} over $\ZZ[1/\rmN d \rmD_{\rmF}]$
\begin{equation*}
\theta: \rmX(\rmB)\otimes\ZZ[1/\rmN d \rmD_{\rmF}]\rightarrow \rmX(\rmQ)
\end{equation*}
by composing the previously defined map with the canonical map $\widetilde{\rmX}(\rmQ)\rightarrow \rmX(\rmQ)$.

Let $(\rmX, \iota, \rho_{\rmX})$ be an element of $\calM_{\mathrm{Dr}}(\rmS)$ with $\rmS$ in $(\mathrm{Nilp})$, then we can restrict $\iota: \calO_{\rmD}\rightarrow \End_{\rmS}(\rmX)$ to ${\ZZ_{p^{2}}}$ and $(\rmX, \iota_{\vert_{\ZZ_{p^{2}}}}, \rho_{\rmX})$ gives rise to an element of 
$\calM(\rmS)$ and hence we can define a morphism $\calM_{\mathrm{Dr}}\rightarrow \calM$. Let $\rmM=\rmM_{0}\oplus\rmM_{1}$ be the Dieudonn\'e lattice of $\rmX$. Suppose that $i\in\{0, 1\}$ is a critical index for $\rmM$ with respect to the action of $\calO_{\rmD}$, then $\rmV\rmM_{i}=\Pi\rmM_{i}$. It follows that $\rmV^{2}\rmM_{i}=p\rmM_{i}$ therefore $i$ is also a critical index for $\rmM$ with respect to the restricted action of $\ZZ_{p^{2}}$. Therefore the morphism $\calM_{\mathrm{Dr}}\rightarrow \calM$ respects the partition with respect to the two notion of critical indices and thus to the $\circ$ and $\bullet$ components, that is we have morphisms
$\calM^{\circ}_{\mathrm{Dr}}\rightarrow \calM^{\circ}$ and  $\calM^{\bullet}_{\mathrm{Dr}}\rightarrow \calM^{\bullet}$. 

\begin{lemma}\label{HZ-bar}
We have the following statements on the reduction of the Hirzebruch--Zagier morphism.
\begin{enumerate}
\item The image of the induced morphism on the special fiber
\begin{equation*}
\overline{\theta}: \overline{\rmX}(\rmB)\rightarrow \overline{\rmX}(\rmQ).
\end{equation*}
is contained in the supersingular locus $\overline{\rmX}^{\ss}(\rmQ)$ of $\overline{\rmX}(\rmQ)$.
\item The map $\overline{\theta}$ respects the partition according to the critical indices and induces a map 
\begin{equation*}
\overline{\theta}^{?}: \overline{\rmX}^{?}(\rmB)\rightarrow \overline{\rmX}^{?}(\rmQ)
\end{equation*}
which restricts to the canonical embedding of 
\begin{equation*}
\vartheta^{?}: \rmZ^{?}(\overline{\rmB})\rightarrow \rmZ^{?}(\overline{\rmQ})
\end{equation*}
for $?\in\{\circ, \bullet\}$.
\end{enumerate}
\end{lemma}
\begin{proof}
The first part $(1)$ follows from Proposition \ref{ss-locus} $(1)$ and \ref{curve-red} $(1)$.  The second part $(2)$ follows from the previous discussions on the compatibilities of the two notions of critical indices. Note a similar result is obtained in \cite{Lan}.
\end{proof}

\subsection{Tate cycles on the quaternionic Shimura surface}
Suppose that $\pi^{\rmQ}$ is a cuspidal automorphic representation of $\rmG(\rmQ)$ of weight $(2, 2)$ and trivial central character.  We will always assume that $\pi^{\rmQ}$ is \emph{non-dihedral} and is defined over $\rmE_{\lambda}$. Suppose that $\pi^{\rmQ}$ appears in the middle degree cohomology  
\begin{equation*}
\rmH^{2}(\rmX(\rmQ)\otimes\overline{\QQ}, \rmE_{\lambda})=\bigoplus\limits_{\Pi^{\rmQ}}\rmH^{2}(\rmX(\rmQ)\otimes\overline{\QQ}, \rmE_{\lambda})[\Pi^{\rmQ}]\otimes (\Pi^{\rmQ\infty})^{\rmK_{\rmN^{+},d}}
\end{equation*}
of $\rmX(\rmQ)\otimes\overline{\QQ}$.  We can associate to $\pi^{\rmQ}$ a two dimensional $\lambda$-adic representation  $(\rho_{\pi^{\rmQ},\lambda}, \rmV_{\pi^{\rmQ}})$ of $\rmG_{\rmF}$ using the construction of Blasius--Rogawski \cite{BR} and Taylor \cite{Tay}. 

The Asai representation associated to the representation $\rmV_{\pi^{\rmQ}}$ will be denoted by
\begin{equation*}
(\mathrm{As}(\rho_{\pi^{\rmQ},\lambda}), \mathrm{As}(\rmV_{\pi^{\rmQ}})),
\end{equation*}
this is a representation of $\rmG_{\QQ}$ isomorphic to the tensor induction $\otimes\mathrm{Ind}^{\rmG_{\QQ}}_{\rmG_{\rmF}}\rmV_{\pi^{\rmQ}}$ of $\rmV_{\pi^{\rmQ}}$ from $\rmG_{\rmF}$ to $\rmG_{\rmQ}$. It will be important for us to note that  $(\mathrm{As}(\rho_{\pi^{\rmQ},\lambda})(-1), \mathrm{As}(\rmV_{\pi^{\rmQ}})(-1))$ is realized on the cohomology $\rmH^{2}(\rmX(\rmQ)\otimes\overline{\QQ}, \rmE_{\lambda}(1))$ given by the same construction as in \cite{BL}.

\begin{lemma}\label{iso-asai}
The $\lambda$-adic Galois representation 
\begin{equation*}
\rmH^{2}_{\pi^{\rmQ}}(\rmX(\rmQ)\otimes{\overline{\QQ}}, \rmE_{\lambda}(1))=\rmH^{2}(\rmX(\rmQ)\otimes\overline{\QQ}, \rmE_{\lambda}(1))[\pi^{\rmQ}] 
\end{equation*}
is isomorphic to $m(\pi^{\rmQ},d)$ copies of $\mathrm{As}(\rmV_{\pi^{\rmQ}})(-1)$ for some $m(\pi^{\rmQ},d)\geq 1$ .
\end{lemma}
\begin{proof}
The same proof for \cite[Lemma 3.9]{Liu-HZ} works here for our quaternionic Hilbert--Blumenthal surface as well noting that $\rmH^{2}_{\pi^{\rmQ}}(\rmX(\rmQ)\otimes{\overline{\QQ}}, \rmE_{\lambda}(1))$ is a semi-simple Galois module by \cite{Nekovar2}. 
\end{proof}

Let $\TT_{\rmF}=\TT^{\frakd\frakn}_{\rmF}$ be the Hecke algebra for the group $\rmG(\rmQ)$ consisting of prime-to-$\frakd\frakn$ Hecke operators where $\frakn$ is the ideal $\rmN\calO_{\rmF}$ and $\frakd$ is the ideal $d\calO_{\rmF}$. The associated Hecke eigensystem of $\pi^{\rmQ}$ are given by the traces of the Frobenius elements of the representation $\overline{\rho}_{\pi^{\rmQ}}$.  Then the traces of the Frobenius elements of the residual representation $\overline{\rho}_{\pi^{\rmQ}}$ of ${\rho}_{\pi^{\rmQ}}$  defines a morphism 
\begin{equation*}
\psi_{\pi^{\rmQ}}: \TT_{\rmF}\rightarrow k_{\lambda} \phantom{aaa} \rmT_{v}\mapsto \mathrm{Tr}\overline{\rho}_{\pi^{\rmQ}}(\mathrm{Fr}_{v}).
\end{equation*}
We define
\begin{equation}
\frakm_{\rmF} =\mathrm{ker}(\psi_{\pi^{\rmQ}}: \TT_{\rmF}\rightarrow k_{\lambda})
\end{equation}
to be the maximal ideal given by the kernel of this morphism and call it the maximal ideal associated to $\overline{\rho}_{\pi^{\rmQ}}$. 

Now we will consider the cohomology of $\rmX(\rmQ)\otimes\overline{\QQ}$ with integral coefficients. To simplify the notations, we will set 
\begin{equation*}
\rmH^{i}_{\pi^{\rmQ}}(\rmX(\rmQ)\otimes\overline{\QQ}, \calO_{\lambda}(j))=\rmH^{i}(\rmX(\rmQ)\otimes\overline{\QQ}, \calO_{{\lambda}}(j))_{\frakm_{\rmF}} 
\end{equation*}
for integers $i, j\geq 0$ and call it the $\pi^{\rmQ}$-isotypic component of $\rmH^{i}(\rmX(\rmQ)\otimes\overline{\QQ}, \calO_{{\lambda}}(j))$. In particular,  $\rmH^{2}_{\pi^{\rmQ}}(\rmX(\rmQ)\otimes\overline{\QQ}, \calO_{\lambda}(1))$ defines an integral lattice in $\mathrm{As}(\rmV_{\pi^{\rmQ}})(-1)$ which we will denote by $\mathrm{As}(\rmT_{\pi^{\rmQ}})(-1)$ which in turn determines an integral lattice $\rmT_{\pi^{\rmQ}}$ of $\rmV_{\pi^{\rmQ}}$.

\begin{definition}
We say the number $d$ that appears in the level structure for $\rmX(\rmQ)$ is {\em clean} for $\rmT_{\pi^{\rmQ}}$ if 
\begin{equation*}
\rmH^{2}_{\pi^{\rmQ}}(\rmX(\rmQ)\otimes\overline{\QQ}, \calO_{\lambda}(1))=m(\pi^{\rmQ},d)\mathrm{As}(\rmT_{\pi^{\rmQ}})(-1)
\end{equation*}
with $m(\pi^{\rmQ},d)$ defined in Lemma \ref{iso-asai}. 
\end{definition}

This terminology is first introduced in \cite{Liu-cubic} from a slightly different situation. Intuitively, this condition means that the $d$-new forms occurring on $\rmH^{2}(\rmX(\rmQ)\otimes\overline{\QQ}, \calO_{\lambda}(1))$ will not be congruent to the $d$-old forms.  From here on, we will always choose $d$ to be clean for $\rmT_{\pi^{\rmQ}}$.  Note that by the proper base change theorem, we can identify
\begin{equation*}
\rmH^{2}_{\pi^{\rmQ}}(\overline{\rmX}(\rmQ)\otimes\overline{\FF}_{p}, \calO_{\lambda}(1))=\rmH^{2}(\overline{\rmX}(\rmQ)\otimes\overline{\FF}_{p}, \calO_{{\lambda}}(1))_{\frakm_{\rmF}} 
\end{equation*}
with $\rmH^{2}_{\pi^{\rmQ}}(\rmX(\rmQ)\otimes\overline{\QQ}_{p}, \calO_{\lambda}(1))$ and thus with $m(\pi^{\rmQ}, d)$-copies of $\mathrm{As}(\rmT_{\pi^{\rmQ}})(-1)\vert_{\rmG_{\QQ_{p}}}$.

Let $\pi^{\overline{\rmQ}}$ be the Jacquet-Langlands transfer of $\pi^{\rmQ}$ to $\rmG(\overline{\rmQ})$. We define similarly the $\pi^{\overline{\rmQ}}$-isotypic component of the cohomology $\rmH^{0}(\rmZ^{?}(\overline{\rmQ}), \calO_{\lambda})=\calO_{\lambda}[\rmZ^{?}(\overline{\rmQ})]$ of the Shimura set $\rmZ(\overline{\rmQ})$ by
\begin{equation*}
\calZ^{?}_{\pi^{\overline{\rmQ}}}(\overline{\rmQ})=\calO_{\lambda}[\rmZ^{?}(\overline{\rmQ})]_{\frakm_{\rmF}}
\end{equation*}
for $?\in \{\circ, \bullet\}$. If we would like to identify these spaces for $?\in\{\circ, \bullet\}$, then we will omit the superscript $?\in\{\circ,\bullet\}$ from the notations. The description of the supersingular locus of $\overline{\rmX}(\rmQ)$ in Lemma \ref{ss-locus} and the cycle class map provides the following map 
\begin{equation*}
\mathcal{C}: \calO_{\lambda}[\rmZ^{\circ}(\overline{\rmQ})]\oplus \calO_{\lambda}[\rmZ^{\bullet}(\overline{\rmQ})]  \rightarrow \rmH^{2}(\overline{\rmX}(\rmQ)\otimes\overline{\FF}_{p}, \calO_{\lambda}(1))^{\mathrm{Fr}_{p^{2}}}
\end{equation*}
which can be considered as a geometric realization of the Jacquet--Langlands correspondence between the group $\rmG(\rmQ)$ and $\rmG(\overline{\rmQ})$.
\begin{definition}
We say the prime $p$ is \emph{admissible} for $\rmT_{\pi^{\rmQ}}$ if its associated Galois representation satisfies the following condition
\begin{enumerate}
\item $p$ is an inert prime in $\rmF$;
\item $a_{p^{2}}(\pi^{\rmQ}) \not\in\{2p, -2p, -p^{2}-1, p^{2}+1\}$ where $a_{p^{2}}(\pi^{\rmQ})=\mathrm{tr}(\mathrm{Fr}_{p^{2}}, \rmT_{\pi^{\rmQ}})$.
\end{enumerate}
We say the prime $p$ is $n$-\emph{admissible} for $\rmT_{\pi^{\rmQ}}$ if it is admissible and 
\begin{equation*}
a_{p^{2}}(\pi^{\rmQ}) \not\in\{2p, -2p, -p^{2}-1, p^{2}+1\} \mod \lambda^{n}.
\end{equation*}
\end{definition}

We now record the following key fact on the Tate conjecture for the special fiber of the quaternionic Shimura surface $\rmX(\rmQ)$ proved in \cite{TX}  which directly inspires our construction of the Flach system on $\rmX(\rmQ)$. 

\begin{proposition}\label{Tate}
Suppose $p$ is an admissible prime for $\rmT_{\pi^{\rmQ}}$. Then the restriction of the map $\mathcal{C}$ induces an isomorphism
\begin{equation*}
\mathcal{C}_{\pi^{\rmQ}}: \calZ^{\circ}_{\pi^{\overline{\rmQ}}}(\overline{\rmQ})\oplus \calZ^{\bullet}_{\pi^{\overline{\rmQ}}}(\overline{\rmQ})  \xrightarrow{\sim} \rmH^{2}_{\pi^{\rmQ}}(\overline{\rmX}(\rmQ)\otimes\overline{\FF}_{p}, \calO_{\lambda}(1))^{\mathrm{Fr}_{p^{2}}}.
\end{equation*}
Moreover, the action of the Frobenius element of $\Gal(\FF_{p^{2}}/\FF_{p})$  on the right hand switches the two factors on the left hand side.
\end{proposition}
\begin{proof}
This follows from the main result of \cite{TX}, see in particular section $1.1$ of \cite{TX} for this particular case. Note as explained in \cite{Liu-HZ}, we can upgrade the result of \cite{TX} to integral coefficient using the assumption that $\mathrm{tr}(\mathrm{Fr}_{p^{2}}, \rmT_{\pi^{\rmQ}}) \mod\lambda \not\in\{2p, -2p\}$.
\end{proof}
\section{Flach classes and reciprocity formula}
\subsection{Flach class of the quaternionic Shimura surface} 
Let $\rmX$ be a proper smooth variety of finite type over a field $\rmK$. For an integer $d$, consider the complex
\begin{equation}\label{M-complex}
\bigoplus_{x\in \rmX^{d-1}}\rmK_{2}k(x)\rightarrow \bigoplus_{x\in \rmX^{d}}k(x)^{\times}\xrightarrow{d_{1}} \bigoplus_{x\in\rmX^{d+1}}\ZZ
\end{equation}
where $\rmX^{i}$ denotes the set of points of codimension $i$ on the variety $\rmX$, $\rmK_{2}k(x)$ is Milnor K-group of the field $k(x)$, the first map is the so-called tame symbol map and the second is given by the divisor map. The \emph{motivic cohomology group} of $\rmX$
\begin{equation*}
\rmH^{2d+1}_{\calM}(\rmX, \ZZ(d+1))
\end{equation*}
is defined to be the cohomology of this complex: elements of $\rmH^{2d+1}_{\calM}(\rmX, \ZZ(d+1))$ are represented by the formal sums $\sum_{i}(\rmZ_{i}, f_{i})$ of pairs consisting of a codimension $d$ cycle $\rmZ_{i}$ on $\rmX$ and a non-zero rational function $f_{i}$ on $\rmZ_{i}$ such that $\sum_{i}\mathrm{div}_{\rmZ_{i}}(f_{i})=0$ as a Weil divisor on $\rmZ_{i}$. This group is also known as the higher Chow group $\mathrm{CH}^{d}(\rmX, 1)$ of $\rmX$ in the literature. There is a Chern character map
\begin{equation}\label{chern}
\mathrm{ch}: \rmH^{2d+1}_{\calM}(\rmX, \ZZ(d+1))\rightarrow \rmH^{2d+1}(\rmX, \ZZ_{l}(d+1))
\end{equation}
given by the coniveau spectral sequences in K-theory and in \'etale cohomology. 

Next suppose that $\calO$ is a complete local ring with fraction field $K$ and residue field $k$ of characteristic $p$ different from $l$. Let $\interX$ be a proper regular scheme over $\calO$, let $\rmX$ be its generic fiber and $\overline{\rmX}$ be its special fiber. The motivic cohomology $\rmH^{2d}_{\calM}(\overline{\rmX}, \ZZ(d))$ in this case agree with the usual Chow group $\mathrm{CH}^{d}(\overline{\rmX})$. There is a map
\begin{equation}\label{div}
\mathrm{div}: \rmH^{2d+1}_{\calM}(\rmX, \ZZ(d+1))\rightarrow \rmH^{2d}_{\calM}(\overline{\rmX}, \ZZ(d))
\end{equation}
defined by sending a pair $(\rmZ, f)$ to the divisor of $f$ on the closure $\mathcal{Z}$ of $\rmZ$ in $\interX$. Note that this divisor is entirely supported on the special fiber $\overline{\calZ}$ of  $\calZ$.

From the definition, an element of the motivic cohomology of $ \rmH^{3}_{\calM}(\rmX(\rmQ)\otimes{\QQ},\ZZ(2))$ can be represented by a curve on the surface $\rmX(\rmQ)\otimes{\QQ}$ and a rational function on this curve with trivial Weil divisors. We will consider the  element  represented by the pair
\begin{equation*}
\Theta^{[p]}=(\theta_{\ast}\rmX(\rmB)\otimes{\QQ}, p)
\end{equation*}
where $\theta_{\ast}\rmX(\rmB)\otimes\QQ$ is the image of $\rmX(\rmB)\otimes\QQ$ under the Hirzebruch--Zagier morphism $\theta$ and we will refer to this element as the {\em Flach element} in the motivic cohomology of the quaternionic Hilbert--Blumenthal surface $\rmX(\rmQ)\otimes\QQ$. 

\begin{proposition}\label{CT}
Suppose that the residual Galois representation $\overline{\rho}_{\pi^{\rmQ}}$ attached to $\pi^{\rmQ}$ has non-solvable image. Then the $\pi^{\rmQ}$-isotypic component 
\begin{equation*}
\rmH^{i}_{\pi^{\rmQ}}(\rmX(\rmQ)\otimes{\overline{\QQ}}, \calO_{\lambda})=\rmH^{i}(\rmX(\rmQ)\otimes{\overline{\QQ}}, \calO_{{\lambda}})_{\frakm_{\rmF}} 
\end{equation*}
of $\rmH^{i}(\rmX(\rmQ)\otimes{\overline{\QQ}}, \calO_{\lambda})$ vanishes unless $i=2$.
\end{proposition}

\begin{proof}
This is the main theorem of \cite{CT}, see \cite[Theorem 7.5.2]{CT}.
\end{proof}

\begin{lemma}
Assume that the residual Galois representation $\overline{\rho}_{\pi^{\rmQ}}$ attached to $\pi^{\rmQ}$ has non-solvable image. Then we have an isomorphism 
\begin{equation*}
\rmH^{3}(\rmX(\rmQ), \calO_{\lambda}(2))_{\frakm_{\rmF}}=\rmH^{1}(\QQ, \rmH^{2}_{\pi^{\rmQ}}(\rmX(\rmQ)\otimes\overline{\QQ}, \calO_{\lambda}(2))).
\end{equation*}
\end{lemma}
\begin{proof}
This follows from the Hoschchild--Serre spectral sequence applied to the cohomology $\rmH^{3}(\rmX(\rmQ)\otimes\QQ, \calO_{{\lambda}}(2))_{\frakm_{\rmF}}$ and the fact that $\rmH^{i}(\rmX(\rmQ)\otimes{\overline{\QQ}}, \calO_{{\lambda}})_{\frakm_{\rmF}}\neq 0$ only for $i=2$.
\end{proof}

The Chern character map in \eqref{chern} induces a map
\begin{equation*}
\mathrm{ch}_{\calO_{\lambda}}: \rmH^{3}_{\calM}(\rmX(\rmQ)\otimes\QQ, \ZZ(2))\rightarrow \rmH^{3}(\rmX(\rmQ)\otimes\QQ, \calO_{\lambda}(2))_{\frakm_{\rmF}}.
\end{equation*}
Suppose the Galois representation $\overline{\rho}_{\pi^{\rmQ}}$ associated to $\frakm_{\rmF}$ has non-solvable image, then it induces the following {\em Abel--Jacobi map}
\begin{equation*}
\mathrm{AJ}_{\pi^{\rmQ}}:  \rmH^{3}_{\calM}(\rmX(\rmQ)\otimes\QQ, \ZZ(2))\rightarrow \rmH^{1}(\QQ, \rmH^{2}_{\pi^{\rmQ}}(\rmX(\rmQ)\otimes\overline{\QQ}, \calO_{\lambda}(2)))
\end{equation*}
in light of the previous lemma. The image of the Flach element $\Theta^{[p]}$ under this map will be denoted by
\begin{equation}\label{Flach-class}
\kappa^{[p]}=\mathrm{AJ}_{\pi^{\rmQ}}(\Theta^{[p]})
\end{equation}
and will be referred to as the {\em Flach class} and it will play a pivotal role for our Euler system argument later. 

\subsection{Local bevaviours of Flach classes} To use the Euler system argument, we will have to analyze the local behaviours of the Flach class $\kappa^{[p]}$ at all primes $r$ of $\QQ$. Let $\rmN(\pi^{\rmQ})$ be the product of primes at which the Asai representation $((\mathrm{As}(\rho_{\pi^{\rmQ},\lambda})(-1), \mathrm{As}(\rmV_{\pi^{\rmQ}})(-1)))$ is ramified.

\begin{lemma}
For any prime $r\nmid \rmN(\pi^{\rmQ})$,  the Flach class is unramified:
\begin{equation*}
\mathrm{loc}_{r}(\kappa^{[p]})\in \rmH^{1}_{\mathrm{fin}}(\QQ_{r}, \rmH^{2}_{\pi^{\rmQ}}(\rmX(\rmQ)\otimes\overline{\QQ}_{r}, \calO_{\lambda}(2))).
\end{equation*}
\end{lemma}
\begin{proof}
Since  $\rmH^{2}_{\pi^{\rmQ}}(\rmX(\rmQ)\otimes\overline{\QQ}_{r}, \calO_{\lambda}(2))$ is unramified at $r$. It follows that
\begin{equation}\label{sin}
\begin{aligned}
\rmH^{1}_{\mathrm{sin}}(\QQ_{r}, \rmH^{2}_{\pi^{\rmQ}}(\rmX(\rmQ)\otimes\overline{\QQ}_{r}, \calO_{\lambda}(2)))
&=(\rmH^{2}_{\pi^{\rmQ}}(\rmX(\rmQ)\otimes\overline{\QQ}_{r}, \calO_{\lambda}(1)))^{\rmG_{\FF_{r}}} \\
&=(\rmH^{2}_{\pi^{\rmQ}}(\overline{\rmX}(\rmQ)\otimes\overline{\FF}_{r}, \calO_{\lambda}(1)))^{\rmG_{\FF_{r}}}. \\
\end{aligned}
\end{equation}
On the other hand, we have a commutative diagram by \cite[Theorem 3.1.1]{Weston1}
\begin{equation*}
\begin{tikzcd}
\rmH^{3}_{\calM}(\rmX(\rmQ)\otimes\QQ_{r},\ZZ(2)) \arrow[r, "\mathrm{AJ}_{\pi^{\rmQ}}"] \arrow[d, "\mathrm{div}_{r}"] & \rmH^{1}(\QQ_{r}, \rmH^{2}_{\pi^{\rmQ}}(\rmX(\rmQ)\otimes\overline{\QQ}_{r}, \calO_{\lambda}(2))) \arrow[d, "\partial_{r}"] \\
\rmH^{2}_{\calM}(\overline{\rmX}(\rmQ)\otimes\FF_{r},\ZZ(1)) \arrow[r, "\mathrm{cl}"]                 &      (\rmH^{2}_{\pi^{\rmQ}}(\overline{\rmX}(\rmQ)\otimes\overline{\FF}_{r}, \calO_{\lambda}(1)))^{\rmG_{\FF_{r}}}  
\end{tikzcd}
\end{equation*}
where 
\begin{enumerate}
\item the top horizontal map is the Abel--Jacobi map $\mathrm{AJ}_{\pi^{\rmQ}}$ for $\rmX(\rmQ)\otimes\QQ_{r}$;
\item the bottom horizontal map is given by the usual cycle class map of the special fiber $\overline{\rmX}(\rmQ)\otimes\FF_{r}$;
\item the right vertical map $\partial_{r}$ is the singular quotient map at $r$ under the identification of \eqref{sin}.
\end{enumerate}
By definition, it is clear that $\mathrm{div}_{r}(\Theta^{[p]})$ vanishes in $\rmH^{2}_{\calM}(\overline{\rmX}(\rmQ)\otimes\FF_{r},\ZZ(1))$, indeed $\mathrm{div}_{r}(\Theta^{[p]})$ is only supported at the special fiber of $\theta_{\ast}\rmX(\rmB)$ at $p$. Thus, $\partial_{r}(\kappa^{[p]})$ vanishes and hence $\mathrm{loc}_{r}(\kappa^{[p]})$ lies in the finite part. 
\end{proof}

Next we turn to the more interesting case of analyzing $\mathrm{loc}_{p}(\Theta^{[p]})$. First we give a slight refinement of Proposition \ref{Tate}.

\begin{lemma}
Suppose $p$ is an admissible prime for $\rmT_{\pi^{\rmQ}}$, then we have an isomorphism
\begin{equation*}
\calZ_{\pi^{\overline{\rmQ}}}(\overline{\rmQ})\xrightarrow{\sim}(\rmH^{2}_{\pi^{\rmQ}}(\overline{\rmX}(\rmQ)\otimes\overline{\FF}_{p}, \calO_{\lambda}(1)))^{\rmG_{\FF_{p}}}.
\end{equation*}
\end{lemma} 
\begin{proof}
By Proposition \ref{Tate}, $(\rmH^{2}_{\pi^{\rmQ}}(\overline{\rmX}(\rmQ)\otimes\overline{\FF}_{p}, \calO_{\lambda}(1)))^{\rmG_{\FF_{p}}}$ can be identified as the invariant space of  $(\rmH^{2}_{\pi^{\rmQ}}(\overline{\rmX}(\rmQ)\otimes\overline{\FF}_{p}, \calO_{\lambda}(1)))^{\rmG_{\FF_{p^{2}}}}$ under the action of $\Gal(\FF_{p^{2}}/\FF_{p})$. Hence it can be identified as the diagonal image of $\calZ_{\pi^{\overline{\rmQ}}}(\overline{\rmQ})$ in $\calZ^{\circ}_{\pi^{\overline{\rmQ}}}(\overline{\rmQ})\oplus\calZ^{\bullet}_{\pi^{\overline{\rmQ}}}(\overline{\rmQ})$ as the non-trivial element of $\Gal(\FF_{p^{2}}/\FF_{p})$ interchanges the $\circ$-component and the $\bullet$-component of the supersingular locus of $\overline{\rmX}^{\mathrm{ss}}(\rmQ)$. The lemma is proved.
\end{proof}

Now we state the \emph{reciprocity formula} for the class $\kappa^{[p]}$. First we define a bilinear pairing 
\begin{equation*}
(\cdot, \cdot): \calO_{\lambda}[\rmZ(\overline{\rmQ})]\times \calO_{\lambda}[\rmZ(\overline{\rmQ})]\rightarrow \calO_{\lambda}
\end{equation*}
on the space $\calO_{\lambda}[\rmZ(\overline{\rmQ})]$ by
\begin{equation*}
(\zeta, \phi)=\sum\limits_{z\in \rmZ(\overline{\rmQ})}\zeta\phi(z).
\end{equation*}
\begin{proposition}\label{reci}
Let $p$ be an admissible prime for $\rmT_{\pi^{\rmQ}}$ and $\phi$ be any element in 
$\calZ_{\pi^{\overline{\rmQ}}}(\overline{\rmQ})$.
Then we have the following identity
\begin{equation*}
(\partial_{p}(\kappa^{[p]}), \phi)=\sum\limits_{z\in \rmZ(\overline{\rmB})}\phi(z)
\end{equation*}
and here the sum is taken over elements in the image of $\rmZ(\overline{\rmB})$ in $\rmZ(\overline{\rmQ})$ via the canonical embedding $\nu$.
\end{proposition}
\begin{proof}
By definition, it is clear that $\mathrm{div}_{p^{2}}(\Theta^{[p]})=[\rmX^{\circ}(\rmB)]+[\rmX^{\bullet}(\rmB)]$ as an element in the motivic  cohomology 
\begin{equation*}
\rmH^{2}_{\calM}(\overline{\rmX}(\rmQ)\otimes\FF_{p^{2}},\ZZ(1))
\end{equation*}
which can be identified with the usual Chow group $\mathrm{CH}^{1}(\overline{\rmX}(\rmQ)\otimes\FF_{p^{2}})$.  
We will still consider the commutative diagram used in the previous case 
\begin{equation}\label{diagram-p}
\begin{tikzcd}
\rmH^{3}_{\calM}(\rmX(\rmQ)\otimes\QQ_{p^{2}},\ZZ(2)) \arrow[r, "\mathrm{AJ}_{\pi^{\rmQ}}"] \arrow[d, "\mathrm{div}_{p^{2}}"] & \rmH^{1}(\QQ_{p^{2}}, \rmH^{2}_{\pi^{\rmQ}}(\rmX(\rmQ)\otimes\overline{\QQ}_{p}, \calO_{\lambda}(2))) \arrow[d, "\partial_{p^{2}}"] \\
\rmH^{2}_{\calM}(\overline{\rmX}(\rmQ)\otimes\FF_{p^{2}},\ZZ(1)) \arrow[r, "\mathrm{cl}"]                 &      (\rmH^{2}_{\pi^{\rmQ}}(\overline{\rmX}(\rmQ)\otimes\overline{\FF}_{p}, \calO_{\lambda}(1)))^{\rmG_{\FF_{p^{2}}}}  
\end{tikzcd}
\end{equation}
and we are concerned with the singular residue $\partial_{p^{2}}(\kappa^{[p]})$. By the commutativity of this diagram,  $\partial_{p^{2}}(\kappa^{[p]})$ agrees with the cycle class of $[\rmX^{\circ}(\rmB)]+[\rmX^{\bullet}(\rmB)]$ in $(\rmH^{2}_{\pi^{\rmQ}}(\overline{\rmX}(\rmQ)\otimes\overline{\FF}_{p}, \calO_{\lambda}(1)))^{\rmG_{\FF_{p^{2}}}}$.

It follows that under the identification 
\begin{equation}\label{cyc-class}
 \calZ^{\circ}_{\pi^{\overline{\rmQ}}}(\overline{\rmQ})\oplus \calZ^{\bullet}_{\pi^{\overline{\rmQ}}}(\overline{\rmQ})  \xrightarrow{\sim} \rmH^{2}_{\pi^{\rmQ}}(\overline{\rmX}(\rmQ)\otimes\overline{\FF}_{p}, \calO_{\lambda}(1))^{\mathrm{Fr}_{p^{2}}},
\end{equation}
$\partial_{p^{2}}(\kappa^{[p]})$ is given by $(\mathbf{1}_{\rmZ^{\circ}(\overline{\rmB})}, \mathbf{1}_{\rmZ^{\bullet}(\overline{\rmB})})$ where $\mathbf{1}_{\rmZ^{?}(\overline{\rmB})}$ is the characteristic function of $\rmZ^{?}(\overline{\rmB})$ for $?\in\{\circ, \bullet\}$. This follows from 
the fact that the isomorphism in \eqref{cyc-class} is given by the cycle class map and statements of Lemma \ref{HZ-bar}.
Thus under the identification
\begin{equation*}
\calZ_{\pi^{\overline{\rmQ}}}(\overline{\rmQ})\xrightarrow{\sim}(\rmH^{2}_{\pi^{\rmQ}}(\overline{\rmX}(\rmQ)\otimes\overline{\FF}_{p}, \calO_{\lambda}(1)))^{\rmG_{\FF_{p}}},
\end{equation*}
the singular residue $\partial_{p}(\kappa^{[p]})$ is simply given by $\mathbf{1}_{\rmZ(\overline{\rmB})}$ and hence we have
\begin{equation*}
\begin{aligned}
(\partial_{p}(\kappa^{[p]}), \phi)&=(\mathbf{1}_{\rmZ(\overline{\rmB})}, \phi)
=\sum\limits_{z\in \rmZ(\overline{\rmB})}\phi(z).\\
\end{aligned}
\end{equation*}
\end{proof}
\section{Distinguished representation and base change}
\subsection{Dinstinguished representation} We recall the definition of a distinguished representation and its quaternionic variant. Let $\Pi$ be a cuspidal automorphic representation of $\GL_{2}(\mathbf{A}_{\rmF})$ with trivial central character, then $\Pi$ is called \emph{distinguished} if there exists an automorphic form $\phi$ in the space of $\Pi$ such that the period
\begin{equation*}
\calP^{\GL_{2}}_{\mathrm{dis}}(\phi)=\int_{\GL_{2}(\QQ)\backslash\GL_{2}(\mathbf{A}_{\rmF})}\phi(g)dg
\end{equation*}
is non-vanishing. It is known that $\pi$ is distinguished if it satisfies the following properties: 
\begin{enumerate}
\item $\Pi_{\infty}$ is in the discrete series;  
\item $\Pi$ is in the image of base-change of an automorphic representation $\pi$ of $\GL_{2}(\mathbf{A})$.
\end{enumerate}

We will consider the quaternionic variant of the above period for the quaternion algebra $\overline{\rmQ}=\overline{\rmB}\otimes \rmF$. Suppose that $\pi^{\overline{\rmQ}}$ is a cuspidal automorphic representation of $\rmG(\overline{\rmQ})(\mathbf{A})$. Then we say $\pi^{\overline{\rmQ}}$ is distinguished if there exists an automorphic form $\phi$ in the space of $\pi^{\overline{\rmQ}}$ such that the period
\begin{equation*}
\calP_{\mathrm{dis}}(\phi)=\int_{\rmG(\overline{\rmB})(\QQ)\backslash\rmG(\overline{\rmB})(\mathbf{A})}\phi(g)dg
\end{equation*}
is non-vanishing. For a fixed $\phi$, we will call $\calP_{\mathrm{dis}}(\phi)$ the distinguished period of $\phi$.

These distinguished representations  appear naturally in the  proof of Tate conjectures for Hilbert--Blumenthal surfaces and its quaternionic variant \cite{HLR}, \cite{Lai}, \cite{FH}. In fact, it can be shown that the Tate classes defined on abelian extensions of $\QQ$ are all supported on the isotypic component of the distinguished representations. We will not use this fact in the sequel.

The following proposition shows that the notion of being distinguished is compatible with respect to the Jacquet--Langlands correspondence.
\begin{proposition}\label{FH}
Suppose that $\Pi$ is a cuspidal automorphic representation of $\GL_{2}(\mathbf{A}_{\rmF})$ which corresponds to a cuspidal automorphic representation $\pi^{\overline{\rmQ}}$ of $\rmG(\overline{\rmQ})$ under the Jacquet--Langlands correspondence. Then $\pi^{\overline{\rmQ}}$ is distinguished with respect to $\rmG(\overline{\rmB})$ if and only if  $\pi$ is distinguished with respect to $\GL_{2}(\mathbf{A})$ and for each prime $r$ dividing the discriminant $\rmD({\overline{\rmB}})$ of $\overline{\rmB}$ which is inert in $\rmF$, the local representation $\pi^{\overline{\rmQ}}_{r}$ is not a principal series rerpesentation $\rmI(\mu_{1}, \mu_{2})$ with $\mu_{i}$ trivial on $\QQ^{\times}_{r}$ for $i=1, 2$.
\end{proposition}
\begin{proof}
This is a special case of the main theorem of Flicker--Hakim, see \cite[Theorem 0.3]{FH}.
\end{proof}

\subsection{Asai representation and base change} In light of the above proposition and our reciprocity formula in Proposition \ref{reci}, we will from here on restrict our attention to the situation given by the following datum:
\begin{enumerate}
\item $f$ is a weight $2$ newform in $\rmS_{2}(\Gamma_{0}(\rmN))$ which defines a cuspidal automorphic representation $\pi$ of $\GL_{2}(\mathbf{A})$; 
\item $\pi$ admits a Jacquet--Langlands transfer to a cuspidal automorphic representation $\pi^{\circ}=\pi^{\overline{\rmB}}$ of $\rmG(\overline{\rmB})(\mathbf{A})$;
\item $\pi^{\overline{\rmQ}}$ is the cuspidal automorphic representation of $\rmG(\overline{\rmQ})(\mathbf{A})$ given by the base change of $\pi^{\overline{\rmB}}$ to $\rmG(\overline{\rmQ})(\mathbf{A})$;
\item $\pi^{\rmQ}$ is the cuspidal automorphic representation of $\rmG({\rmQ})(\mathbf{A})$ obtained from $\pi^{\overline{\rmQ}}$ via the Jacquet--Langlands transfer from $\pi^{\overline{\rmQ}}$ to $\rmG({\rmQ})(\mathbf{A})$.
\end{enumerate}

Let $(\rho_{\pi^{\circ}}, \rmV_{\pi^{\circ}})$ be the Galois representation attached to $\pi^{\circ}$ and $(\rho_{\pi^{\rmQ}}, \rmV_{\pi^{\rmQ}})$ be the Galois representation attached to $\pi^{\rmQ}$. By construction, the Galois representation attached to $\pi^{\overline{\rmQ}}$ is the same as $(\rho_{\pi^{\rmQ}}, \rmV_{\pi^{\rmQ}})$. In the situation above, the Asai representation $\mathrm{As}(\rmV_{\pi^{\rmQ}})(-1)$ can be related to the representation $\rmV_{\pi^{\circ}}$ as in the following lemma. 
\begin{lemma}
We have an isomorphism 
\begin{equation*}
\mathrm{As}(\rmV_{\pi^{\rmQ}})(-1)\cong \mathrm{Sym}^{2}(\rmV_{\pi^{\circ}})(-1)\oplus \rmE_{\lambda}(\omega_{\rmF/\QQ}).
\end{equation*}
Note that $\mathrm{Sym}^{2}(\rmV_{\pi^{\circ}})(-1)$ is also isomorphic to $\mathrm{Ad}^{0}(\rmV_{\pi^{\circ}})$. 
Similarly, we also have a such decomposition integrally
\begin{equation*}
\mathrm{As}(\rmT_{\pi^{\rmQ}})(-1)\cong \mathrm{Sym}^{2}(\rmT_{\pi^{\circ}})(-1)\oplus \calO_{\lambda}(\omega_{\rmF/\QQ})
\end{equation*}
for the lattice $\mathrm{Sym}^{2}(\rmT_{\pi^{\circ}})(-1)$ in $\mathrm{Sym}^{2}(\rmV_{\pi^{\circ}})(-1)$.
\end{lemma}

\section{Bounding the adjoint Selmer groups}
\subsection{Generalities on Selmer groups}
We will consider a general Galois representation $\rho: \rmG_{\QQ}\rightarrow\GL(\rmV)$ over $\rmE_{\lambda}$. Suppose that $\rmT$ is a Galois stable lattice in $\rmV$ and set $\calM=\rmV/\rmT$, this is called the divisible Galois module associated to $\rho$. We recall some relevant definitions and facts concerning the Bloch--Kato Selmer group for $\calM$. These Galois modules fit in the exact sequence 
\begin{equation*}
0\rightarrow \rmT\xrightarrow{i} \rmV\xrightarrow{\mathrm{pr}} \calM\rightarrow 0.
\end{equation*}
Let $\rmM_{n}=\calM[\lambda^{n}]$ and $\rmT_{n}=\rmT/\lambda^{n}\rmT$. Let $i_{n}: \rmM_{n}\hookrightarrow \calM$ and $\mathrm{pr}_{n}:\rmT\rightarrow \rmT_{n}$ be the natural inclusion and reduction maps. The Galois module  $\rmM_{n}$ will be referred to as the finite Galois module.

We define the local Bloch--Kato conditions using the following recipe:
\begin{enumerate}
\item for $v\neq l$, we define $\rmH^{1}_{f}(\QQ_{v},\rmV)=\rmH^{1}_{\mathrm{fin}}(\QQ_{v}, \rmV)$;
\item for $v=l$, we define $\rmH^{1}_{f}(\QQ_{p},\rmV)=\mathrm{ker}\{\rmH^{1}(\QQ_{l}, \rmV)\rightarrow \rmH^{1}(\QQ_{l}, \rmV\otimes \rmB_{\mathrm{cris}})\}$;
\item we define $\rmH^{1}_{f}(\QQ_{r},\calM)=\mathrm{pr}_{\ast}\rmH^{1}_{f}(\QQ_{r}, \rmV)$ for each prime $r$;  
\item we define $\rmH^{1}_{f}(\QQ_{r},\rmM_{n})=i^{\ast}_{n}\rmH^{1}_{f}(\QQ_{r}, \calM)$  for each prime $r$.
\end{enumerate}
We will be interested in the Bloch--Kato Selmer group of $\calM$ defined by
\begin{equation*}
\rmH^{1}_{f}(\QQ, \calM)=\mathrm{ker}\{\rmH^{1}(\QQ,\calM)\rightarrow \prod_{r}\frac{\rmH^{1}(\QQ_{r}, \calM)}{\rmH^{1}_{f}(\QQ_{r},\calM)}\} .
\end{equation*}
The Bloch--Kato Selmer group of $\rmM_{n}$ is defined by the same recipe replacing $\calM$ by $\rmM_{n}$.
Moreover we have
\begin{equation*}
\rmH^{1}_{f}(\QQ, \calM)=\lim\limits_{\longrightarrow}\rmH^{1}_{f}(\QQ, \rmM_{n})
\end{equation*}
and an exact sequence 
\begin{equation*}
0\rightarrow \mathrm{pr}_{\ast}\rmH^{1}_{f}(\QQ, \rmV)\rightarrow \rmH^{1}_{f}(\QQ, \calM)\rightarrow \Sha(\QQ, \calM)\rightarrow 0
\end{equation*}
defining the \emph{Tate-Shafarevich group} $\Sha(\QQ, \calM)$ for $\calM$. The proof of following lemma is an easy diagram chase and therefore will be omitted.

\begin{lemma}
Suppose we have an exact sequence of divisible or finite Galois modules
\begin{equation*}
0\rightarrow \calM^{\prime}\rightarrow \calM\rightarrow \calM^{\prime\prime}\rightarrow 0.
\end{equation*}
Then it induces an exact sequence
\begin{equation*}
\begin{aligned}
0\rightarrow \rmH^{0}(\QQ, \calM^{\prime})\rightarrow \rmH^{0}(\QQ, \calM)\rightarrow \rmH^{0}(\QQ, \calM^{\prime\prime})\rightarrow \rmH^{1}_{f}(\QQ, \calM^{\prime})\rightarrow \rmH^{1}_{f}(\QQ, \calM)\rightarrow \rmH^{1}_{f}(\QQ, \calM^{\prime\prime})\rightarrow 0.\\
\end{aligned}
\end{equation*}
\end{lemma}

\subsection{Statement of the main result}
We now consider the datum given in the beginning of \S 5.2.  Let $f\in\rmS_{2}(\Gamma_{0}(\rmN))$ be a newform of level $\rmN$ whose associated automorphic representation of $\GL_{2}(\Adel)$ is given by $\pi$. Suppose that $\rmN$ admits a decomposition $\rmN=\rmN^{+}\rmN^{-}$. Under the assumption that $\rmN^{-}$ is square-free and consists of odd number of prime divisors, $f$ admits a normalized Jacquet--Langlands transfer  $f^{\dagger}$ as an automorphic form realized in $\calO_{\lambda}[\rmZ(\overline{\rmB})]$. Here normalized means that there is an element $z\in \rmZ(\overline{\rmB})$ such that $f^{\dagger}(z)$ is non-zero modulo $\lambda$. The base change of $f^{\dagger}$ to $\rmF$ can be regarded as an automorphic form realized in $\calO_{\lambda}[\rmZ(\overline{\rmQ})]$ which we will denote by $\phi^{\dagger}$. We will choose $\phi^{\dagger}$ such that it is normalized as well in the same sense for $f^{\dagger}$. In the discussion of \S 5.2, the automorphic form $f^{\dagger}$ is contained in an automorphic representation of $\rmG(\overline{\rmB})(\mathbf{A})$ which is denoted by $\pi^{\circ}=\pi^{\overline{\rmB}}$ and $\phi^{\dagger}$ is  contained in the automorphic representation $\pi^{\overline{\rmQ}}$ of $\rmG(\overline{\rmQ})(\mathbf{A})$ which is the base change of $\pi^{\circ}$ to $\rmF$. Recall also that $\pi^{\overline{\rmQ}}$ admits a Jacquet--Langlands transfer $\pi^{\rmQ}$ to $\rmG(\rmQ)(\mathbf{A})$.  

For the convenience of the reader, we recall some notations concerning Galois modules and period integrals that we will use later:\begin{enumerate}
\item Let $(\rho_{\pi^{\circ}}, \rmV_{\pi^{\circ}})$ be the Galois representation associated to $\pi^{\circ}$ and $\rmT_{\pi^{\circ}}$ be the Galois stable lattice that we have fixed in $\S 5.2$. Then we define the divisible Galois module $\calM_{\pi^{\circ}}$ by
\begin{equation*}
0\rightarrow \rmT_{\pi^{\circ}}\rightarrow \rmV_{\pi^{\circ}}\rightarrow \calM_{\pi^{\circ}}\rightarrow 0.
\end{equation*}

\item Let $(\rho_{\pi^{\rmQ}}, \rmV_{\pi^{\rmQ}})$ be the Galois representation associated to $\pi^{\rmQ}$ and $\rmT_{\pi^{\rmQ}}$ be the Galois stable lattice we have fixed in $\S 5.2$. Then we define the divisible Galois module $\calM_{\pi^{\rmQ}}$ by
\begin{equation*}
0\rightarrow \rmT_{\pi^{\rmQ}}\rightarrow \rmV_{\pi^{\rmQ}}\rightarrow \calM_{\pi^{\rmQ}}\rightarrow 0.
\end{equation*}

\item We have a natural decomposition 
\begin{equation*}
\mathrm{As}(\calM_{\pi^{\rmQ}})(-1)=\mathrm{Sym}^{2}(\calM_{\pi^{\circ}})(-1)\oplus\rmE_{\lambda}/\calO_{\lambda}(\omega_{\rmF/\QQ}).
\end{equation*}

\item Note that $\mathrm{Sym}^{2}(\calM_{\pi^{\circ}})(-1)$ is the same as $\mathrm{Ad}^{0}(\calM_{\pi^{\circ}})$. We will set 
\begin{equation*}
\rmM_{n}=\mathrm{Ad}^{0}(\calM_{\pi^{\circ}}[\lambda^{n}]) 
\end{equation*}
and 
\begin{equation*}
\rmN_{n}=\mathrm{Sym}^{2}(\rmT_{\pi^{\circ}, n}) 
\end{equation*}
for $\rmT_{\pi^{\circ}, n}=\rmT_{\pi^{\circ}}\mod \lambda^{n}$ and $n\geq 1$.

\item Recall we have an embedding $
\vartheta: \rmZ(\overline{\rmB})\rightarrow \rmZ(\overline{\rmQ})$ of Shimura sets as in Lemma \ref{HZ-bar}
and thus we can consider the integral distinguished period defined by
\begin{equation*}
\calP_{\mathrm{dis}}(\phi^{\dagger})=\sum\limits_{z\in \rmZ(\overline{\rmB})}\phi^{\dagger}(z)
\end{equation*}
where the sum is taken over elements of the image of $\vartheta:  \rmZ(\overline{\rmB})\rightarrow \rmZ(\overline{\rmQ})$.  
\end{enumerate}

The following theorem will be the main theorem on bounding the Selmer group using the distinguished period. 
\begin{theorem}\label{main1}
Let $f\in \rmS_{2}(\Gamma_{0}(\rmN))$ be a newform of weight $2$ with $\rmN=\rmN^{+}\rmN^{-}$ such that $\rmN^{-}$ is squarefree and has odd number of prime factors. Let $f^{\dagger}$ be the normalized automorphic form in $\calO_{\lambda}[\rmZ(\overline{\rmB})]$ corresponding to $f$ under the Jacquet--Langlands correspondence and $\phi^{\dagger}$ be the base-change of $f^{\dagger}$ which is also normalized. Let $\nu=\mathrm{ord}_{\lambda}(\calP_{\mathrm{dis}}(\phi^{\dagger}))$ and $\eta=\varpi^{\nu}$. Suppose the following conditions hold.
\begin{enumerate}
\item The residual Galois representation $\overline{\rho}_{\pi^{\circ}}$ is absolutely irreducible;
\item The image of $\overline{\rho}_{\pi^{\circ}}$ contains $\GL_{2}(\FF_{p})$;
\item We further assume that $\rmH^{1}(\Delta_{n}, \rmM_{n})=0$ 
for every $n\geq 1$ where $\Delta_{n}=\Gal(\QQ(\rmM_{n})/\QQ)$ for the splitting field $\QQ(\rmM_{n})$ of the Galois module $\rmM_{n}$.
\end{enumerate}
Then $\eta$ annihilates the Selmer group $\rmH^{1}_{f}(\QQ, \mathrm{Ad}^{0}(\calM_{\pi^{\circ}}))$, moreover we have 
\begin{equation*}
\mathrm{length}_{\calO_{\lambda}}\phantom{.}\rmH^{1}_{f}(\QQ, \mathrm{Ad}^{0}(\calM_{\pi^{\circ}}))\leq \nu.
\end{equation*}
\end{theorem}
 
\subsection{The Flach system argument} 
To prove the main theorem, we will show that under the assumptions of the main theorem, $\eta$ annihilates each finite Selmer group $\rmH^{1}_{f}(\QQ, \rmM_{n})$ with $n\geq1$. Here the argument follows closely that of \cite{Flach} and \cite{Weston1}. We define 
\begin{equation*}
\rmH^{2}_{\pi^{\rmQ}}(\rmX(\rmQ)\otimes\overline{\QQ},\calO_{n}(2))
\end{equation*}
to be the reduction of $\rmH^{2}_{\pi^{\rmQ}}(\rmX(\rmQ)\otimes\overline{\QQ},\calO_{\lambda}(2))$ modulo $\lambda^{n}$.
\begin{lemma}
For each $n\geq 1$, the Galois module
\begin{equation*}
\rmH^{2}_{\pi^{\rmQ}}(\rmX(\rmQ)\otimes\overline{\QQ},\calO_{n}(1))
\end{equation*}
is isomorphic to $m(\pi^{\rmQ}, d)$ copies of $\mathrm{Sym}^{2}(\rmT_{\pi^{\circ}, n})(-1)\oplus\calO_{n}(\omega_{\rmF/\QQ})$. In particular, there is a projection from 
$\rmH^{2}_{\pi^{\rmQ}}(\rmX(\rmQ)\otimes\overline{\QQ},\calO_{n}(2))$
to $\rmN_{n}=\mathrm{Sym}^{2}(\rmT_{\pi^{\circ}, n})$.
\end{lemma}
\begin{proof}
Recall that we have an isomorphism 
\begin{equation*}
\rmH^{2}_{\pi^{\rmQ}}(\rmX(\rmQ)\otimes\overline{\QQ},\calO_{\lambda}(1))\cong m(\pi^{\rmQ}, d)(\mathrm{Sym}^{2}\rmT_{\pi^{\circ}}(-1)\oplus \calO_{\lambda}(\omega_{\rmF/\QQ}))
\end{equation*}
by the definition of $d$-cleaness of $\rmT_{\pi^{\rmQ}}$.
The lemma follows from this immediately from Proposition \ref{CT} which implies that $\rmH^{2}_{\pi^{\rmQ}}(\rmX(\rmQ)\otimes\overline{\QQ},\calO_{\lambda}(1))$ is torsion free.
\end{proof}
We consider the Flach class
\begin{equation*}
\kappa^{[p]}\in \rmH^{1}(\QQ, \rmH^{2}_{\pi^{\rmQ}}(\rmX(\rmQ)\otimes\overline{\QQ},\calO_{\lambda}(2)))
\end{equation*}
defined as in \ref{Flach-class} and we will use this class to construct annihilator of the Selmer group $\rmH^{1}_{f}(\QQ, \mathrm{Ad}^{0}(\calM_{\pi^{\circ}})$. By the above lemma, we can project this class to $\rmH^{1}(\QQ, \rmN_{n})$ and we will denote by $\kappa^{[p]}_{n}$ the resulting class. We also recall that $\eta=\varpi^{\nu}$ where $\nu=\ord_{\lambda}(\calP_{\mathrm{dis}}(\phi^{\dagger}))$.

\begin{lemma}\label{index}
Let $p$ be an $n$-admissible prime for $\rmT_{\pi^{\rmQ}}$. 
\begin{enumerate}
\item The singular quotient $\rmH^{1}_{\sin}(\QQ_{p}, \rmN_{n})$ is free of rank one
over $\calO_{n}$. 
\item The quotient of $\rmH^{1}_{\sin}(\QQ_{p}, \rmN_{n})$ by the module generated by $\partial_{p}(\kappa^{[p]}_{n})$ is annihlated by $\eta$.
\end{enumerate}
\end{lemma}
\begin{proof}
By the definition of an $n$-admissible prime, it is easily calculated that among the Frobenius eigenvalues at $p$ for $\rmN_{n}$, the number of times the value $1$ can only appear once. Since we have 
\begin{equation*}
\rmH^{1}_{\sin}(\QQ_{p}, \rmN_{n})=\Hom(\ZZ_{l}(1),  \rmN_{n})^{\rmG_{\FF_{p}}}=(\rmN_{n}(-1))^{\rmG_{\FF_{p}}}, 
\end{equation*}
$\rmH^{1}_{\sin}(\QQ_{p}, \rmN_{n})$ is free of rank one over $\calO_{n}$.  It follows from the reciprocity law proved in Proposition \ref{reci} that the quotient of $\rmH^{1}_{\sin}(\QQ_{p}, \rmN_{n})$ by the module generated by $\partial_{p}(\kappa^{[p]}_{n})$ is annihilated by $\eta$. 
\end{proof}

\begin{lemma}\label{lm1}
Let $p$ be an $n$-admissible prime for $\rmT_{\pi^{\rmQ}}$ and define 
\begin{equation*}
\rmH^{1}_{\{p\}}(\QQ,\rmM_{n})=\mathrm{ker}\{\rmH^{1}(\QQ,\rmM_{n})\rightarrow \rmH^{1}(\QQ_{p},\rmM_{n})\}.
\end{equation*}
Then we have 
\begin{equation*}
\eta\rmH^{1}_{f}(\QQ, \rmM_{n})\subset \rmH^{1}_{\{p\}}(\QQ,\rmM_{n}).
\end{equation*}
\end{lemma}
\begin{proof}
We consider the element $\kappa^{[p]}_{n}$ in $\rmH^{1}(\QQ,\rmN_{n})$. We have verified that $\mathrm{loc}_{r}(\kappa^{[p]}_{n})$ lies in the finite part for all $r\nmid p\rmN$. On the other hand, the same proofs for \cite[Lemma 2.8]{Flach} and \cite[Lemma 2.10]{Flach} carry over here and shows that $\mathrm{loc}_{r}(\kappa^{[p]}_{n})$ lie in $ \rmH^{1}_{f}(\QQ_{r},\rmM_{n})$ for $r\mid \rmN$. Under the local Tate duality
\begin{equation*}
\langle\cdot,\cdot\rangle_{r}: \rmH^{1}(\QQ_{r},\rmN_{n})\times \rmH^{1}(\QQ_{r},\rmM_{n})\rightarrow \calO_{n},
\end{equation*}
it is well--known that our local conditions $\rmH^{1}_{f}(\QQ_{r},\rmN_{n})$ are orthogonal to  $\rmH^{1}_{f}(\QQ_{r},\rmM_{n})$ and at $r=p$ induces a perfect pairing 
\begin{equation*}
\langle\cdot,\cdot\rangle_{p}: \rmH^{1}_{\mathrm{sin}}(\QQ_{p},\rmN_{n})\times \rmH^{1}_{\mathrm{fin}}(\QQ_{p},\rmM_{n})\rightarrow \calO_{n}.
\end{equation*}
Let $s$ be any element in $\rmH^{1}_{f}(\QQ, \rmM_{n})$, then by global class field theory
\begin{equation*}
\sum_{r}\langle \mathrm{loc}_{r}(s), \mathrm{loc}_{r}(\kappa^{[p]})\rangle_{r}=0.
\end{equation*}
This identity reduces to $\langle\mathrm{loc}_{p}(s), \mathrm{loc}_{p}(\kappa^{[p]})\rangle_{p}=0$ by the discussions above. It follows that $\langle\mathrm{loc}_{p}(s), \partial_{p}(\kappa^{[p]})\rangle_{p}=0$. Since we know that $\eta\rmH^{1}_{\sin}(\QQ_{p}, \rmN_{n})$ is contained in the line generated by $\partial_{p}\kappa^{[p]}$ by Lemma \ref{index}, $\eta\mathrm{loc}_{p}(s)$ has to vanish by the perfectness of the above pairing. Therefore $\eta\rmH^{1}_{f}(\QQ, \rmM_{n})\subset \rmH^{1}_{\{p\}}(\QQ,\rmM_{n})$ follows from this.
\end{proof}

\begin{lemma}\label{lm2}
Let $\rmF_{n}=\QQ(\rmM_{n})$ be the splitting field for the Galois module $\rmM_{n}$ and $\Delta_{n}=\Gal(\rmF_{n}/\QQ)$. Then we have
\begin{equation*}
\rmH^{1}_{\{p\}}(\QQ, \rmM_{n})\subset \rmH^{1}(\Delta_{n}, \rmM_{n})
\end{equation*}
where $\rmH^{1}(\Delta_{n}, \rmM_{n})$ is considered as a subgroup of $\rmH^{1}(\QQ, \rmM_{n})$ via inflation.
\end{lemma}
\begin{proof}
Let $s\in \rmH^{1}_{\{p\}}(\QQ, \rmM_{n})$ and consider the exact sequence 
\begin{equation}
0\rightarrow \rmH^{1}(\Delta_{n}, \rmM_{n})\rightarrow \rmH^{1}(\QQ, \rmM_{n})\rightarrow \rmH^{1}(\rmF_{n}, \rmM_{n})^{\Delta_{n}}=\Hom_{\Delta_{n}}(\rmG_{\rmF_{n}}, \rmM_{n})\rightarrow0.
\end{equation}
Let $\psi: \rmG_{\rmF_{n}}\rightarrow \rmM_{n}$ be the image of $s$ in $\Hom_{\Delta_{n}}(\rmG_{\rmF_{n}}, \rmM_{n})$. Then we need to show that $\psi=0$. Let $\tilde{s}$ be the cocycle representing $s$ and let $\rmF^{\prime}_{n}$ be the fixed field of the kernel of $\tilde{s}$. Let $\Gamma$ be $\Gal(\rmF^{\prime}_{n}/\rmF_{n})$, then it is clear that $\psi$ factors through $\psi: \Gamma\rightarrow \rmM_{n}$. Let $\tau$ be the Frobenius element at $p$ in $\Delta_{n}$ and fix a lift $\tau^{\prime}$ to $\Gal(\rmF^{\prime}_{n}/\QQ)$. Let $g$ be any element in $\Gamma$. By the Chebatorev density theorem we can find a place $v^{\prime}$ of $\rmF^{\prime}_{n}$ such that $\mathrm{Fr}_{\rmF^{\prime}_{n}/\rmF_{n}}(v^{\prime})=\tau^{\prime}g$. Let $v$ be the place under $v^{\prime}$ in $\rmF_{n}$ which necessarily lies over $p$. 

Since $s_{p}:=\mathrm{res}_{p}(s)$ is trivial, $\tilde{s}\vert_{\Gal(\rmF^{\prime}_{n, v^{\prime}}/\QQ_{p})}$ is a coboundary. Thus we have
\begin{equation*}
\tilde{s}(\tau^{\prime}g)\in (\tau^{\prime}g-1)\rmM_{n}=(\tau-1)\rmM_{n}. 
\end{equation*} 
Taking $g=1$ gives $\tilde{s}(\tau^{\prime})\in (\tau-1)\rmM_{n}$. On the other hand, the cocycle relation gives
\begin{equation*}
\tilde{s}(\tau^{\prime}g)=\tilde{s}(\tau^{\prime})+\tau\tilde{s}(g).
\end{equation*}
It follows then $\tau\tilde{s}(g)\in (\tau-1)\rmM_{n}$. Hence $\tilde{s}(g)\in (\tau-1)\rmM_{n}$ for any $g\in\Gamma$ as $(\tau-1)\tilde{s}(g)\in (\tau-1)\rmM_{n}$. Thus the image of $\psi$ lie in $(\tau-1)\rmM_{n}$. Note that $\psi$  is $\Delta_{n}$-equivariant and $\rmM_{n}$ is irreducible. Since $\rmM_{n}\neq (\tau-1)\rmM_{n}$ by the same reasoning as in Lemma \ref{index}, the image of $\psi$ is zero. 
\end{proof}

\begin{lemma}\label{lm3}
Under the same assumptions of Theorem \ref{main1}, $\eta$ annihilates the Selmer group $\rmH^{1}_{f}(\QQ, \mathrm{Ad}^{0}(\calM_{\pi^{\circ}}))$. Moreover $\rmH^{1}_{f}(\QQ, \mathrm{Ad}^{0}(\calM_{\pi^{\circ}}))$ can be identified with $\rmH^{1}_{f}(\QQ, \rmM_{\nu})$.
\end{lemma}
\begin{proof}
For each $n$, we pick an $n$-admissible prime $p$ for $\rmT_{\pi^{\rmQ}}$. By the previous Lemma \ref{lm1} and Lemma \ref{lm2}, we have
\begin{equation*}
\eta\rmH^{1}_{f}(\QQ, \rmM_{n}) \subset \rmH^{1}(\Delta_{n}, \rmM_{n}).
\end{equation*}
By the second assumption in the statement of the Theorem, we know $\rmH^{1}(\Delta_{n}, \rmM_{n})$ is trivial. Thus $\rmH^{1}_{f}(\QQ, \rmM_{n})$ is indeed annihilated by $\eta$ for each $n$. Hence $\rmH^{1}_{f}(\QQ, \Ad^{0}(\calM_{\pi^{\circ}}))$ is also annihilated by $\eta$ as desired. 

For the in particular part, we abbreviate $\mathrm{Ad}^{0}(\calM_{\pi^{\circ}})$ by $\rmM$. Consider the exact sequence 
\begin{equation*}
0\rightarrow \rmM_{n}\rightarrow \rmM\xrightarrow{\varpi^{n}}\varpi^{n}\rmM\rightarrow 0. 
\end{equation*}
It induces a long exact sequence
\begin{equation*}
0\rightarrow\rmH^{0}(\QQ, \rmM_{n})\rightarrow \rmH^{0}(\QQ, \rmM)\rightarrow \rmH^{0}(\QQ, \varpi^{n}\rmM)\rightarrow \rmH^{1}_{f}(\QQ, \rmM_{n})\rightarrow  \rmH^{1}_{f}(\QQ, \rmM)\rightarrow\rmH^{1}_{f}(\QQ, \varpi^{n}\rmM).
\end{equation*}
The assumptions in the statement of the lemma implies that there is an isomorphism
\begin{equation*}
\rmH^{1}_{f}(\QQ, \rmM_{n})\cong  \rmH^{1}_{f}(\QQ, \rmM)[\varpi^{n}].
\end{equation*}
It follows then we have $\rmH^{1}_{f}(\QQ, \rmM_{\nu})\cong  \rmH^{1}_{f}(\QQ, \rmM)[\varpi^{\nu}]\cong \rmH^{1}_{f}(\QQ, \Ad^{0}(\calM_{\pi^{\circ}}))$. 
\end{proof}

Next we will use the Bockstein pairing to upgrade the annihilation result above to a result concerning the length of the Selmer group. We define the first  $p$-Tate--Shafarevich group of $\rmM_{n}$ by
\begin{equation*}
\Sha^{1}_{\{p\}}(\QQ, \rmM_{n})=\ker\{\rmH^{1}_{f}(\QQ, \rmM_{n})\rightarrow \rmH^{1}(\QQ_{p}, \rmM_{n})\}.
\end{equation*}
It is clear that $\Sha^{1}_{\{p\}}(\QQ, \rmM_{n})\subset \rmH^{1}_{\{p\}}(\QQ, \rmM_{n})$. 

Suppose that $\Sha^{1}_{\{p\}}(\QQ, \rmM_{n})$ is trivial, then one can define the \emph{Bockstein pairing}
\begin{equation*}
\{\cdot, \cdot\}_{n}: \rmH^{1}_{f}(\QQ, \rmN_{n})\otimes \rmH^{1}_{f}(\QQ, \rmM_{n})\rightarrow \QQ_{l}/\ZZ_{l}
\end{equation*}
associated to the exact sequence
\begin{equation*}
0\rightarrow \rmN_{n}\xrightarrow{\varpi^{n}} \rmN_{2n}\rightarrow  \rmN_{n}\rightarrow 0
\end{equation*}
where is the first map is the multiplication by $\varpi^{n}$ and the second map is the natural projection map. We will not recall the definition of this pairing but refer the reader to \cite{Weston1} for some discussions.

Let $p$ be a $\nu$-admissible prime for $\rmT_{\pi^{\rmQ}}$ now.  Consider the submodule $\langle\kappa^{p}_{\nu}\rangle$ of $\rmH^{1}(\QQ, \rmN_{\nu})$ generated by $\kappa^{p}_{\nu}$, we observe that $\langle\kappa^{p}_{\nu}\rangle$ lies in $ \rmH^{1}_{f}(\QQ, \rmN_{\nu})$ by Lemma \ref{index}.

\begin{lemma}\label{lm4}
Under the assumptions of Theorem \ref{main1}, the restriction of $\{\cdot,\cdot\}_{\nu}$ to $\langle\kappa^{p}_{\nu}\rangle$ is right non-degenerate and it induces an inclusion
\begin{equation*}
\rmH^{1}_{f}(\QQ, \mathrm{Ad}^{0}(\calM_{\pi^{\circ}}))\hookrightarrow \Hom(\langle\kappa^{p}_{\nu}\rangle, \QQ_{l}/\ZZ_{l}).
\end{equation*} 
\end{lemma}
\begin{proof}
Let $y\in \rmH^{1}_{f}(\QQ,\rmM_{\nu})$. First note that $\Sha^{1}_{\{p\}}(\QQ, \rmM_{\nu})$ is trivial by Lemma \ref{lm2}. Therefore it makes sense to compute the Bockstein pairing $\{\kappa^{p}_{\nu}, y\}_{\nu}$. By definition, this is done by choosing a lift $\kappa^{p}_{2\nu}\in  \rmH^{1}(\QQ, \rmN_{2\nu}) $ of  $\kappa^{p}_{\nu}$ and defining 
\begin{equation*}
\{\kappa^{p}_{\nu}, y\}_{\nu}=\langle\frac{1}{\eta}\partial_{p}(\kappa^{p}_{2\nu}), \mathrm{loc}_{p}(y) \rangle_{\nu}
\end{equation*}
using the Tate local duality at $p$. 

Suppose now $y$ is chosen that $\{\kappa^{p}_{\nu}, y\}_{\nu}=0$. By Lemma \ref{index}, the class $\frac{1}{\eta}\partial_{p}(\kappa^{p}_{2\nu})$ generates $ \rmH^{1}_{\sin}(\QQ_{p}, \rmN_{\nu})$. By the perfectness of the Tate local duality, we have $\loc_{p}(y)=0$ and it follows that $y\in \rmH^{1}_{\{p\}}(\QQ, \rmM_{\nu})$ and thus $y=0$ by
Lemma \ref{lm2}. This shows that we have an inclusion 
\begin{equation*}
\rmH^{1}_{f}(\QQ, \rmM_{\nu})\hookrightarrow \Hom(\langle\kappa^{p}_{\nu}\rangle, \QQ_{l}/\ZZ_{l}).
\end{equation*}
Thus we have $\rmH^{1}_{f}(\QQ, \mathrm{Ad}^{0}(\calM_{\pi^{\circ}}))\hookrightarrow \Hom(\langle\kappa^{p}_{\nu}\rangle, \QQ_{l}/\ZZ_{l})$ by Lemma \ref{lm3}.
\end{proof}

\begin{myproof}{Theorem}{\ref{main1}}
Now Theorem \ref{main1} follows from combining Lemma \ref{lm3} and Lemma \ref{lm4} which gives the desired length estimate
\begin{equation*}
\mathrm{length}_{\calO_{\lambda}}\phantom{.}\rmH^{1}_{f}(\QQ, \mathrm{Ad}^{0}(\calM_{\pi^{\circ}}))\leq \nu.
\end{equation*}
\end{myproof}

\subsection{Comparison of quaternionic periods} 
Recall the set up at the beginning of $\S6.2$, $f\in\rmS_{2}(\Gamma_{0}(\rmN))$ is a modular form which is new at primes dividing $\rmN^{-}$ and admits a normalized Jacquet--Langlands transfer  $f^{\dagger}$ to an automorphic form on $\rmZ(\overline{\rmB})$. The base change of $f^{\dagger}$ to $\rmF$ can be regarded as an automorphic form on $\rmZ(\overline{\rmQ})$ which we denote by $\phi^{\dagger}$. 

The purpose of this final section is to compare the distinguished period
\begin{equation*}
\calP_{\mathrm{dis}}(\phi^{\dagger})=\sum\limits_{z\in \rmZ(\overline{\rmB})}\phi^{\dagger}(z) 
\end{equation*}
and the \emph{quaternionic period}
\begin{equation*}
\calP(f^{\dagger})=\sum\limits_{z\in \rmZ(\overline{\rmB})}f^{\dagger2}(z) =\langle f^{\dagger}, f^{\dagger}\rangle
\end{equation*}
which is known as the Petersson norm of $f^{\dagger}$. We would like to compare $\ord_{\lambda}(\calP_{\mathrm{dis}}(\phi^{\dagger}))$ and $\ord_{\lambda}(\calP(f^{\dagger}))$ under suitable assumptions on the Galois representation $\rho_{\pi^{\circ}}: \rmG_{\QQ}\rightarrow \GL_{2}(\rmE_{\lambda})$ associated to $\pi^{\circ}$.  More precisely, we will show that the $\lambda$-adic valuation of the ratio
\begin{equation*}
\calP_{\mathrm{dis}}(\phi^{\dagger})/\calP(f^{\dagger})
\end{equation*}
is non-negative. Results of this flavour have  been by Prasana in \cite[Theorem 5.2]{pras} and by Urban--Tilouine in \cite[Proposition 4.22]{UT} with quite different method.

Let $\Sigma^{+}$ be the set of primes dividing $\rmN^{+}$ and let $\Sigma^{-}_{\mathrm{ram}}$ be the set of primes $r$ dividing $\rmN^{-}$ such that $l\mid r^{2}-1$.

\begin{assumption}\label{ass1}
We make the following assumptions on $\bar{\rho}_{\pi^{\circ}}$:
\begin{enumerate}
\item  $\bar{\rho}_{\pi^{\circ}}\vert_{\rmG_{\QQ(\zeta_{l})}}$ is absolutely irreducible;
\item The image of $\bar{\rho}_{\pi^{\circ}}$ contains $\GL_{2}(\FF_{l})$;
\item $\overline{\rho}_{\pi^{\circ}}$ is \emph{minimal} at primes in $\Sigma^{+}$ in the sense that all the liftings of $\overline{\rho}_{\pi^{\circ}}\vert_{\rmG_{\QQ_{r}}}$ are minimally ramified for $r\in\Sigma^{+}$;
\item $\overline{\rho}_{\pi^{\circ}}$ is ramified at primes in $\Sigma^{-}_{\mathrm{ram}}$.
\end{enumerate}
\end{assumption}

Now we can state the main result on comparing the periods $\calP_{\mathrm{dis}}(\phi^{\dagger})$ and $\calP(f^{\dagger})$.

\begin{theorem}\label{main2}
Let $f\in \rmS_{2}(\Gamma_{0}(\rmN))$ be a newform of weight $2$ with $\rmN=\rmN^{+}\rmN^{-}$ such that $\rmN^{-}$ is squarefree and has odd number of prime factors. Let $f^{\dagger}$ be the normalized automorphic form on $\rmZ(\overline{\rmB})$ corresponding to $f$ under the Jacquet--Langlands correspondence and $\phi^{\dagger}$ be the base change of $f^{\dagger}$ considered as an automorphic form on  $\rmZ(\overline{\rmQ})$ which is also normalized. 
\begin{enumerate}
\item We assume that the residual Galois representation $\overline{\rho}_{\pi^{\circ}}$ satisfies Assumption \ref{ass1}; 
\item We further assume that $\rmH^{1}(\Delta_{n}, \rmM_{n})=0$ 
for every $n\geq 1$ where $\Delta_{n}=\Gal(\QQ(\rmM_{n})/\QQ)$ for the splitting field $\QQ(\rmM_{n})$ of the Galois module $\rmM_{n}$. 
\end{enumerate}
Then we have the following inequality
\begin{equation*}
\ord_{\lambda}(\calP_{\mathrm{dis}}(\phi^{\dagger}))\geq \ord_{\lambda}(\calP(f^{\dagger})).
\end{equation*}
\end{theorem}

To prove this theorem, we will study the Bloch--Kato Selmer group $\rmH^{1}_{f}(\QQ, \mathrm{Ad}^{0}(\calM_{\pi^{\circ}}))$ from the perspective of deformation theory of the residual representation $\overline{\rho}_{\pi^{\circ}}$. Although the Bloch--Kato Selmer group itself has less  connection with the deformation theory of the residual representation $\overline{\rho}_{\pi^{\circ}}$, we can introduce a smaller Selmer group as follows. We define the local condition $\rmH^{1}_{\mathrm{new}}(\QQ_{v}, \Ad^{0}(\calM_{\pi^{\circ}}))$ and $\rmH^{1}_{\mathrm{new}}(\QQ_{v}, \rmM_{n})$ as in \cite[Definition 3.6]{Lun}, see also \cite[(3.3), (3.4)]{KO}. Let $\Sigma^{-}_{\mathrm{mix}}$ be the set of prime divisors of $\rmN^{-}$such that $l\nmid r^{2}-1$. Then we define the Selmer group $\rmH^{1}_{\calS}(\QQ, \Ad^{0}(\calM_{\pi^{\circ}}))$  by
\begin{equation*}
\mathrm{ker}\{\rmH^{1}(\QQ,\Ad^{0}(\calM_{\pi^{\circ}}))\rightarrow \prod_{v\not\in \Sigma^{-}_{\mathrm{mix}}}\frac{\rmH^{1}(\QQ_{v}, \Ad^{0}(\calM_{\pi^{\circ}}))}{\rmH^{1}_{f}(\QQ_{v},\Ad^{0}(\calM_{\pi^{\circ}}))}\times \prod_{v\in \Sigma^{-}_{\mathrm{mix}}}\frac{\rmH^{1}(\QQ_{v}, \Ad^{0}(\calM_{\pi^{\circ}}))}{\rmH^{1}_{\new}(\QQ_{v},\Ad^{0}(\calM_{\pi^{\circ}}))}\}.
\end{equation*}
Let $\TT=\TT^{\rmN}$ be the prime-to-$\rmN$ Hecke algebra for the group $\rmG(\overline{\rmB})$. Let 
\begin{equation*}
\phi_{\pi^{\circ}}: \TT\rightarrow k_{\lambda}
\end{equation*}
be the morphism provided by the Hecke eigensystem corresponding to the trace of Frobenius of $\overline{\rho}_{\pi^{\circ}}$. Then we obtain a maximal ideal $\frakm=\ker(\phi_{\pi^{\circ}})$. On the other hand, consider the global deformation problem given by
\begin{equation*}
\calS_{\mix}:=(\overline{\rho}_{\pi^{\circ}}, \chi_{l}, \Sigma^{+}\cup \Sigma^{-}_{\ram}\cup\Sigma^{-}_{\mix}\cup\{l\}, \{\calD_{v}\}_{v\in \Sigma^{+}\cup \Sigma^{-}_{\ram}\cup\Sigma^{-}_{\mix}\cup\{l\}})
\end{equation*}
 that classifies the deformations of $\overline{\rho}_{\pi^{\circ}}$ over an $\calO$-algebra which satisfy the following local deformation conditions:
\begin{enumerate}
\item for $v=l$, $\calD_{l}$ classifies deformations which are Fontaine--Laffaille crystalline;
\item for $v\in \Sigma^{-}_{\ram}\cup \Sigma^{+}$, $\calD_{v}$ is the local deformation problem that classifies deformations that are minimally ramified;
\item for $v\in \Sigma^{-}_{\mix}$,  $\calD_{v}=\calD^{\new}_{v}$ is the local deformation problem that classifies deformations that are new in the sense of \cite[Definition 3.6]{Lun}.
\end{enumerate}

This global deformation problem is represented by the deformation ring $\rmR_{\mix}$. Moreover it follows from \cite[Proposition 4.1]{Lun}, see also \cite[Theorem 3.14]{KO} that there is an isomorphism
\begin{equation}
\rmR_{\mix}\cong \TT_{\frakm}.
\end{equation}
We remark that although we use \cite{Lun}  and \cite{KO} as references. This type of $\rmR=\rmT$ theorem is first studied in \cite{Khare} reversing the steps taken in the original Taylor--Wiles method. Consider the congruence number $\eta_{f}(\rmN^{+},\rmN^{-})$ defined in \cite[\S2.2]{PW} that detects congruences between $f$ and modular forms in $\rmS_{2}(\Gamma_{0}(\rmN))$ which are new at primes dividing $\rmN^{-}$, the $\lambda$-valuation $\ord_{\lambda}(\eta_{f}(\rmN^{+},\rmN^{-}))$ agrees with the length of $\rmH^{1}_{\calS}(\QQ, \Ad^{0}(\calM_{\pi^{\circ}}))$ as a consequence of the above $\rmR=\rmT$ theorem. It also follows that $\calZ_{\pi^{\circ}}(\overline{\rmB})=\calO_{\lambda}[\rmZ_{\rmN^{+}}(\overline{\rmB})]_{\frakm}
$ is free over $\TT_{\frakm}$ of rank $1$. Hence \cite[4.17]{DDT} implies that $\eta_{f}(\rmN^{+},\rmN^{-})$ can be chosen to be the Petersson norm $\calP(f^{\dagger})$ of $f^{\dagger}$. Thus we have
\begin{equation*}
\mathrm{length}_{\calO_{\lambda}}\phantom{.}\rmH^{1}_{\calS}(\QQ, \Ad^{0}(\calM_{\pi^{\circ}}))=\ord_{\lambda}(\eta_{f}(\rmN^{+},\rmN^{-}))=\ord_{\lambda}(\calP(f^{\dagger})).
\end{equation*}
On the other hand, we have an inclusion
\begin{equation*}
\rmH^{1}_{\calS}(\QQ, \Ad^{0}(\calM_{\pi^{\circ}}))\hookrightarrow \rmH^{1}_{f}(\QQ, \Ad^{0}(\calM_{\pi^{\circ}}))
\end{equation*}
by the definitions of the local conditions defining the Selmer groups:  indeed, one can show that
\begin{equation*}
\rmH^{1}_{\new}(\QQ_{v},\Ad^{0}(\calM_{\rho}))
\end{equation*}
is trivial by \cite[Proposition 4.4]{KO}. Hence we have
\begin{equation*}
\mathrm{length}\phantom{.}\rmH^{1}_{f}(\QQ, \Ad^{0}(\calM_{\pi^{\circ}}))\geq \mathrm{length}\phantom{.}\rmH^{1}_{\calS}(\QQ,\Ad^{0}( \calM_{\pi^{\circ}}))=\ord_{\lambda}(\calP(f^{\dagger}))
\end{equation*}
\begin{remark}
In fact, we have proved in the companion work \cite{Wang} that 
\begin{equation*}
\mathrm{length}\phantom{.}\rmH^{1}_{f}(\QQ, \Ad^{0}(\calM_{\pi^{\circ}}))=\ord_{\lambda}(\calP(f^{\dagger}))
\end{equation*}
using a similar method as in this article.
\end{remark}

\begin{myproof}{Theorem}{\ref{main2}}
By the main theorem \ref{main1}, we have the following inequality
\begin{equation*}
\ord_{\lambda}(\calP_{\mathrm{dis}}(\phi^{\dagger}))\geq \mathrm{length}\phantom{.}\rmH^{1}_{f}(\QQ, \mathrm{Ad}^{0}(\calM_{\pi^{\circ}})).
\end{equation*}
On the other hand, the above discussion implies that
\begin{equation*}
\mathrm{length}\phantom{.}\rmH^{1}_{f}(\QQ, \Ad^{0}(\calM_{\pi^{\circ}}))\geq \mathrm{length}\phantom{.}\rmH^{1}_{\calS}(\QQ, \Ad^{0}(\calM_{\pi^{\circ}}))=\ord_{\lambda}(\calP(f^{\dagger})).
\end{equation*}
We have arrived at the desired inequality $\ord_{\lambda}(\calP_{\mathrm{dis}}(\phi^{\dagger}))\geq \ord_{\lambda}(\calP(f^{\dagger}))$.
\end{myproof}

\end{document}